\pdfoutput=1
\documentclass[twoside]{article}

\usepackage[preprint]{aistats2026}

\usepackage[utf8]{inputenc} %
\usepackage{cmap}
\usepackage[T1]{fontenc}    %

\let\PolishL\L   %

\usepackage{multicol}
\usepackage{silence}
\usepackage{amsthm}
\usepackage{amsmath,amsfonts,amssymb,thmtools}
\usepackage{mathtools}
\usepackage{xparse}
\usepackage{xargs}
\usepackage{enumitem} %
\usepackage{etoolbox}
\usepackage{hhline}
\usepackage{mathrsfs}	
\usepackage{dsfont}
\usepackage{tikz}			
\usepackage{mathtools}
\usepackage{complexity}
\usepackage{svg}
\usepackage{nag} 					%
\usepackage{arydshln} %
\usepackage{booktabs} %
\usepackage{soul}
\usepackage{setspace}
\usepackage{xurl} %

\usepackage[linktocpage=true, colorlinks=true, allcolors=blue, linktocpage]{hyperref}

\usepackage[capitalise,nameinlink,noabbrev]{cleveref} %
\crefname{equation}{}{}

\newtheorem{example}{Example}
\newtheorem{theorem}{Theorem}
\newtheorem{lemma}[theorem]{Lemma}
\newtheorem{proposition}[theorem]{Proposition}
\newtheorem{remark}[theorem]{Remark}

\crefname{enumi}{Statement}{Statements} %

\makeatletter
\newcommand*\ie{\textit{i.e.}\@ifnextchar.{\@gobble}{\relax}}
\newcommand*\vs{\textit{vs.}\@ifnextchar.{\@gobble}{\relax}}
\newcommand*\etc{\textit{etc.}\@ifnextchar.{\@gobble}{\relax}}
\newcommand*\eg{\textit{e.g.}\@ifnextchar.{\@gobble}{\relax}}
\newcommand*\cf{\textit{cf.}\@ifnextchar.{\@gobble}{\relax}}
\makeatother

\newcommand{\norm}[1]{\| #1 \|} 
 
\newcommand{\abs}[1]{\lvert #1 \rvert} 
\NewDocumentCommand{\infnorm}{s O{} m }{%
  \IfBooleanTF{#1}{\norm*{#3}}{\norm[#2]{#3}}_{\infty}%
}
\NewDocumentCommand{\twonorm}{s O{} m }{%
  \IfBooleanTF{#1}{\norm*{#3}}{\norm[#2]{#3}}_2%
}
\NewDocumentCommand{\tvnorm}{s O{} m }{%
  \IfBooleanTF{#1}{\norm*{#3}}{\norm[#2]{#3}}_{\textup{TV}}%
}
\NewDocumentCommand{\onenorm}{s O{} m }{%
  \IfBooleanTF{#1}{\norm*{#3}}{\norm[#2]{#3}}_1%
}
\NewDocumentCommand{\frobnorm}{s O{} m }{%
  \IfBooleanTF{#1}{\norm*{#3}}{\norm[#2]{#3}}_F%
}
\NewDocumentCommand{\scalar}{s O{} >{\SplitArgument{1}{,}}m}{%
    \IfBooleanTF{#1}{\scalaraux*#3}{\scalaraux[#2]#3}%
}
\DeclarePairedDelimiterX{\scalaraux}[2]{\langle}{\rangle}{#1, #2}

\newcommand*\circledaux[1]{\tikz[baseline=(char.base)]{
    \node[shape=circle,draw,inner sep=0.8pt] (char) {#1};}}

\NewDocumentCommand{\circled}{m o }{%
    \IfNoValueTF{#2}{\circledaux{#1}}{\stackrel{\circledaux{#1}}{#2}}%
}

\usepackage{color}
\definecolor{mygreen}{rgb}{0,0.6,0}
\definecolor{mygray}{rgb}{0.5,0.5,0.5}
\definecolor{mymauve}{rgb}{0.58,0,0.82}
\definecolor{darkgray}{rgb}{0.35,0.35,0.35}

\newcommand{\defi}{\stackrel{\mathrm{\scriptscriptstyle def}}{=}}
\newcommand{\defiin}{\stackrel{\mathrm{\scriptscriptstyle def}}{\in}}
\renewcommand*\R{\mathbb{R}}                              %
\let\epsilon\varepsilon

\def\dodeffunction#1:#2->#3;#4\relax
    {\ifblank{#4}
    {#1\colon#2\to#3}
    {\!\begin{aligned}#1\colon#2&\to#3\\ \dodeffunctionaux#4\relax\end{aligned}}}
\def\dodeffunctionaux#1->#2\relax{#1&\mapsto#2}

\newcommand\dual[1]{#1^\ast}                                %
\usepackage{soul}

\NewDocumentCommand{\enorm}{s O{} m}{%
    \IfBooleanTF{#1}{\norm*{#3}}{\norm[#2]{#3}}_{\E}%
}
\NewDocumentCommand{\denorm}{s O{} m}{%
    \dual{\IfBooleanTF{#1}{\enorm*{#3}}{\enorm[#2]{#3}}}%
}

\NewDocumentCommand{\commasaux}{m m}{%
    \IfNoValueTF{#2}{#1}{#1, #2}%
}
\NewDocumentCommand{\prodaux}{m m}{%
    \IfNoValueTF{#2}{#1 \times #1}{#1 \times #2}%
}

\renewcommand{\E}[1]{\mathbb{E}\left[ #1 \right]} 
\renewcommand{\L}{\mathscr{L}}

\DeclareMathOperator*{\argmin}{arg\,min}                %

\makeatletter
\renewcommand\paragraph{\@startsection{paragraph}{4}{\z@}%
                                    {0ex \@plus0.5ex \@minus.2ex}%
                                    {-1em}%
                                    {\normalfont\normalsize\bfseries}}
\makeatother

\newcommand{\innp}[1]{\langle #1 \rangle}

\newcommand\blfootnote[1]{%
\begingroup
\renewcommand\thefootnote{}\footnote{#1}%
\addtocounter{footnote}{-1}%
\endgroup
}
\usepackage[natbib, backend=biber, maxcitenames=3, maxbibnames=99, style=alphabetic, hyperref, backref, useprefix=true, uniquename=false, doi=false,url=false,eprint=false]{biblatex} 

\usepackage{csquotes}               
\bibliography{refs}        

\DeclareCiteCommand{\cite}
  {\usebibmacro{prenote}}
  {\usebibmacro{citeindex}%
   \printtext[bibhyperref]{\usebibmacro{cite}}}
  {\multicitedelim}
  {\usebibmacro{postnote}}

\DeclareCiteCommand*{\cite}
  {\usebibmacro{prenote}}
  {\usebibmacro{citeindex}%
   \printtext[bibhyperref]{\usebibmacro{citeyear}}}
  {\multicitedelim}
  {\usebibmacro{postnote}}

\DeclareCiteCommand{\parencite}[\mkbibparens]
  {\usebibmacro{prenote}}
  {\usebibmacro{citeindex}%
    \printtext[bibhyperref]{\usebibmacro{cite}}}
  {\multicitedelim}
  {\usebibmacro{postnote}}

\DeclareCiteCommand*{\parencite}[\mkbibparens]
  {\usebibmacro{prenote}}
  {\usebibmacro{citeindex}%
    \printtext[bibhyperref]{\usebibmacro{citeyear}}}
  {\multicitedelim}
  {\usebibmacro{postnote}}

\DeclareCiteCommand{\citeauthor}
  {\usebibmacro{prenote}}
  {\ifciteindex
     {\indexnames{labelname}}
     {}%
   \printtext[bibhyperref]{\printnames{labelname}}}
  {\multicitedelim}
  {\usebibmacro{postnote}}

\DeclareCiteCommand{\footcite}[\mkbibfootnote]
  {\usebibmacro{prenote}}
  {\usebibmacro{citeindex}%
  \printtext[bibhyperref]{ \usebibmacro{cite}}}
  {\multicitedelim}
  {\usebibmacro{postnote}}

\DeclareCiteCommand{\footcitetext}[\mkbibfootnotetext]
  {\usebibmacro{prenote}}
  {\usebibmacro{citeindex}%
   \printtext[bibhyperref]{\usebibmacro{cite}}}
  {\multicitedelim}
  {\usebibmacro{postnote}}

\DeclareCiteCommand{\textcite}
  {\boolfalse{cbx:parens}}
  {\usebibmacro{citeindex}%
   \printtext[bibhyperref]{\usebibmacro{textcite}}}
  {\ifbool{cbx:parens}
     {\bibcloseparen\global\boolfalse{cbx:parens}}
     {}%
   \multicitedelim}
  {\usebibmacro{textcite:postnote}}

\newbibmacro{string+doiurlisbn}[1]{%
  \iffieldundef{doi}{%
    \iffieldundef{url}{%
      \iffieldundef{isbn}{%
        \iffieldundef{issn}{%
          #1%
        }{%
          \href{http://books.google.com/books?vid=ISSN\thefield{issn}}{#1}%
        }%
      }{%
        \href{http://books.google.com/books?vid=ISBN\thefield{isbn}}{#1}%
      }%
    }{%
      \href{\thefield{url}}{#1}%
    }%
  }{%
    \href{https://doi.org/\thefield{doi}}{#1}%
  }%
}

\DeclareFieldFormat{title}{\usebibmacro{string+doiurlisbn}{\mkbibemph{#1}}}
\DeclareFieldFormat[article,incollection,inproceedings]{title}%
    {\usebibmacro{string+doiurlisbn}{\mkbibquote{#1}}}

\usepackage{algorithm}
\usepackage{algcompatible}

\AtBeginEnvironment{algorithmic}{\setstretch{1.15}} 

\algnewcommand{\lst}{\texttt{lst}}
\algnewcommand{\slst}{\texttt{slst}}
\algnewcommand{\SEND}{\textbf{send}}

\newsavebox{\algleft}
\newsavebox{\algright}

\makeatletter
\newcounter{algorithmicH}
\let\oldalgorithmic\algorithmic
\renewcommand{\algorithmic}{%
  \stepcounter{algorithmicH}
  \oldalgorithmic}
\renewcommand{\theHALG@line}{ALG@line.\thealgorithmicH.\arabic{ALG@line}}
\makeatother



\input{includes/definitions.tex}

\begin{document}

\twocolumn[

\aistatstitle{Smooth Quasar-Convex Optimization with Constraints}

\aistatsauthor{David Martínez-Rubio}

\aistatsaddress{IMDEA Software Institute \\
  Madrid, Spain \\
  \texttt{\href{mailto:david.martinezrubio@imdea.org}{david.martinezrubio@imdea.org}} } ]

\begin{abstract}
    Quasar-convex functions form a broad nonconvex class with applications to linear dynamical systems, generalized linear models, and Riemannian optimization, among others. Current nearly optimal algorithms work only in affine spaces due to the loss of one degree of freedom when working with general convex constraints. Obtaining an accelerated algorithm that makes nearly optimal $\bigotilde{1/(\gamma\sqrt{\epsilon})}$ first-order queries to a $\gamma$-quasar convex smooth function \emph{with constraints} was independently asked as an open problem in \citet{martinezrubio2022global,lezane2024accelerated}.
In this work, we solve this question by designing an inexact accelerated proximal point algorithm that we implement using a first-order method achieving the aforementioned rate and, as a consequence, we improve the complexity of the accelerated geodesically Riemannian optimization solution in \citet{martinezrubio2022global}.
We also analyze projected gradient descent and Frank-Wolfe algorithms in this constrained quasar-convex setting. To the best of our knowledge, our work provides the first analyses of first-order methods for quasar-convex smooth functions with general convex constraints.
\blfootnote{\color{darkgray}Most of the non-local notations in this work have a link to their definitions, using \href{https://damaru2.github.io/general/notations_with_links/}{this code}, such as ${\protect\hyperlink{def:quasar_convexity_constant}{\color{darkgray}\oldgamma}}$, which links to where this notation is defined as the quasar-convexity parameter of the functions we optimize in this work.}
\end{abstract}

\section{INTRODUCTION}

Nonconvex optimization has become central to modern machine learning, yet our theoretical understanding of why simple first-order methods succeed in these settings remains limited. Algorithms such as stochastic gradient descent routinely achieve both efficient optimization and strong statistical performance, even though the underlying problems are nonconvex and, in principle, could exhibit a highly irregular landscape with spurious local minima or saddle points. This disconnect motivates a line of research seeking structural properties of nonconvex problems that explain the empirical successes of local-search algorithms.

Classical theory is built around convex optimization, where every local minimum is global, and where powerful tools of duality are available. However, the binary distinction between convex and nonconvex problems is too coarse: many important nonconvex problems exhibit benign structure that allows for efficient optimization. A rich body of work has introduced relaxations of convexity that retain the global optimality of local minima, including star-convexity, essential strong convexity, restricted secant inequalities, one-point strong convexity, variational coherence, quasiconvexity, pseudoconvexity, invexity, the Polyak-\PolishL{}ojasiewicz (PL) condition, tilted convexity, and the strict-saddle property without spurious local minima \citep{karimi2016linear,ge2015escaping,hinderworkshop2020nearoptimal,martinezrubio2022global}.

Besides, it has been shown that all local minimizers are global in a number of machine learning tasks such as  problems in phase retrieval \citep{sun2016geometric}, tensor decomposition \citep{ge2015escaping}, dictionary learning \citep{sun2015complete}, matrix sensing \citep{bhojanapalli2016global,park2016nonsquare}, and matrix completion \citep{ge2016matrixcompletion}. Moreover, under overparameterization assumptions, gradient descent provably finds global minimizers in neural networks \citep{allen2019convergence,du2019global,nguyen2020global,zou2018sgd,du2018provably}. These results highlight a growing recognition that the landscape of many machine learning objectives is more favorable than worst-case nonconvexity suggests and motivates the study of optimization under benign nonconvexity.

Among the relaxations of convexity, \emph{quasar convexity} has recently emerged as a particularly compelling property. It is a generalization of star-convexity, where the slope of the classical affine lower bound given by a gradient and its function value is multiplied by a number greater than $1$, and one only requires for this to bound the function at a specific solution, cf. \cref{sec:preliminaries_and_setting}. This class of functions allows for acceleration under smooth objectives, in the \emph{unconstrained case} and the essentially equivalent case where there is an affine constraint. When the problem has general constraints, current solutions lose one degree of freedom, and current accelerated solutions and analyses do not work.

Important examples of quasar-convex functions include linear dynamical systems \citep{hardt2018gradient}, several generalized linear models (\newtarget{def:acronym_generalized_linear_model}{\GLM{}s}) \citep{foster2018uniform,wang2023continuized}, and geodesically convex problems in constant-curvature Riemannian spaces after appropriate reductions \citep{martinezrubio2022global}. Further, the landscape of neural networks has been observed to empirically satisfy quasar-convexity along the trajectory of gradient descent methods with respect to the reached solution \citep{kleinberg2018alternative,zhou2019sgd,tran2024empirical}.

Despite the progress in the study of optimization with quasar convexity, constrained optimization in this regime remains poorly understood. %
Whether acceleration is possible by first-order methods for smooth quasar-convex problems with general convex constraints has been posed as an open question in independent works \citep{martinezrubio2022global,lezane2024accelerated}, where a fast method in the presence of a compact convex constrained would improve the accelerated Riemannian optimization method in \citep{martinezrubio2022global}. In this paper, we resolve this question in the affirmative by providing an accelerated \emph{constrained} first-order method that matches the lower bound of \citep{hinderworkshop2020nearoptimal}, up to logarithmic factors. Our results extend the scope of quasar-convex optimization to restricted domains, thereby broadening its applicability to machine learning problems with natural structural restrictions. To the best of our knowledge, we provide the first analyses of first-order methods for smooth quasar-convex problems with general constraints.

\paragraph{Main Contributions.}
Our contributions can be summarized as follows:
\begin{enumerate}
    \item We design an accelerated \emph{constrained} first-order method for $\L$-smooth $\gamma$-quasar-convex functions with respect to a minimizer $x^\ast$ in a compact convex feasible set of diameter $\D$, finding an $\epsilon$-minimizer in the nearly optimal %
$\bigotildel{\gamma^{-1}\sqrt{\L\D^2 / \epsilon}}$
        queries to a first-order oracle, where $x_0$ is an initial point, solving the open question in \citep{martinezrubio2022global,lezane2024accelerated}. The solution involved designing an implementable inexact accelerated proximal method along with the design of a line search under adversarial noise due to the inexactness in the proximal solution.
    \item Our algorithm implies improved complexity over the Riemannian optimization solution in \citep{martinezrubio2022global}, compared to prior work on tilted convexity, we achieve faster acceleration and under weaker assumptions.
    \item We show that projected gradient descent as well as for the Frank-Wolfe algorithm converge at the unaccelerated rate $\bigo{\L\D^2/(\gamma^2\epsilon)}$ for $\L$-smooth $\gamma$-quasar convex functions with constraints. 
\end{enumerate}

\section{PRELIMINARIES AND SETTING}\label{sec:preliminaries_and_setting}

\paragraph{Notation.}  
We denote by $\newtarget{def:approximate_argmin}{\argmindelta[x \in \X]} f(x)$ the set of $\delta$-minimizers of $f$ over a feasible set $\newtarget{def:feasible_set}{\X}$. A function is said to be differentiable on a closed (possibly non-open) set $\X$ if is differentiable in an open neighborhood of $\X$. The set indicator function is $\newtarget{def:set_indicator}{\indicator{X}}(x) = 0$ if $x\in\X$ and $\indicator{X}(x) = +\infty$ otherwise. We use $\newtarget{def:big_o_tilde}{\bigotilde{\cdot}}$ as the big-O notation omitting logarithmic factors. For a set of points, we denote its convex hull by $\newtarget{def:convex_combination}{\conv{S}}$. We denote the Euclidean projection operator by $\newtarget{def:projection_operator}{\proj{\X}}(x) \defi \argmin_{y\in\X} \norm{y-x}_2$.

A first-order oracle for $f$ is an operator that, given a query point $x \in \R^d$, returns the pair $(f(x), \nabla f(x))$.

\paragraph{$\gamma$-quasar convexity.} For $\newtarget{def:quasar_convexity_constant}{\gamma} \in (0, 1]$, a differentiable function is said to be $\gamma$-quasar convex in $\X$ with center $x^\ast \in \X$ if $x^\ast \in \argmin_{x\in\X} f(x)$ and for every $x \in \X$: 
\[
    f(x^\ast) \;\geq\; f(x) + \frac{1}{\gamma} \innp{\nabla f(x), x^\ast - x}.
\]
If $\gamma = 1$, the function is said to be star convex. Note that by the definition above, any point $x$ satisfying $\nabla f(x) = 0$ is also a minimizer.

\paragraph{$\L$-smoothness.}  
A differentiable function $f$ has $\newtarget{def:smoothness_constant}{\L}$-Lipschitz gradients in $\X$ if
\[
    \norm{\nabla f(x) - \nabla f(y)}_2 \leq \L \norm{x - y}_2, \qquad \forall x,y \in \X.
\]
Equivalently, $f$ satisfies
\[
    \abs{f(x) - f(y) - \innp{\nabla f(y), x-y}} \;\leq\; \frac{\L}{2}\norm{x-y}_2^2,
\]
for all $x,y \in \X$. This property is also often called $\L$-smoothness, more commonly in convex scenarios, where the absolute value above is redundant. Lastly, we say that a differentiable function $f$ is $\newtarget{def:strong_convexity}{\mu}$-strongly convex in $\X$, for $\mu > 0$, if for all $x,y \in \X$, we have
\[
    f(x) - f(y) - \innp{\nabla f(y), x-y} \geq \frac{\mu}{2}\norm{x-y}_2^2.
\]

\paragraph{Problem setting.}  
We study the problem
\begin{equation}\label{eq:problem_setting}
    \min_{x \in \X} f(x),
\end{equation}
for methods with access to a first-order oracle of $f$, where $\X \subset \R^d$ is a compact convex set of diameter $\newtarget{def:diameter_of_feasible_set}{\D}$, and $f:\R^d \to \R$ is both $\L$-smooth and $\gamma$-quasar convex in $\X$ with center denoted by $x^\ast \in \X$. Our aim is to obtain an $\newtarget{def:accuracy_epsilon}{\epsilon}$-minimizer $\hat{y}$ of the problem above, that is, $f(\hat{y}) - \min_{x\in\X}f(x) \leq \epsilon$.

A relaxation of convexity stronger than $\gamma$-quasar convexity is $(\gamma, \gamma_p)$-tilted convexity. For $\gamma, \gamma_p \in (0, 1]$ and all $x, y$ in a closed convex set $\X$, a tilted-convex function satisfies
\begin{align*}  \label{eq:tiltedconvexity}
 \begin{aligned}
     f(x) +  \frac{1}{\gamma}\innp{\nabla f(x), y-x}\leq f(y)  & {\text{ if } \innp{\nabla f(x), y-x}}\leq 0, \\
     f(x)+\gamma_p\innp{\nabla f(x), y-x} \leq f(y)  & {\text{ if } \innp{\nabla f(x), y-x}} \geq 0.
   \end{aligned}
\end{align*}
This is clearly stronger than $\gamma$-quasar convexity since the first line above includes the quasar-convex property if $x = x^\ast$. We note that \citep[Theorem 5]{martinezrubio2022global} provided a first-order method for optimizing a $(\gamma, \gamma_p)$-tilted convex function with $\L$-Lipschitz gradients in a compact convex set $\X$ of diameter $\D$ in time 
$\bigotilde{\sqrt{\L\D^2/(\gamma^2 \gamma_p \epsilon)}}$,
where the complexity contains a logarithmic dependence on the Lipschitz constant of $f$ in $\X$. Our main result, cf. \cref{thm:acc_constrained_quasar_cvx}, shows that under the weaker assumption of $\gamma$-quasar convexity, we obtain faster convergence $\bigotilde{\sqrt{\L\D^2/(\gamma^2 \epsilon)}}$, where the dependence of the complexity on $\gamma_p$ and the logarithm of the Lipschitz constant of $f$ disappear. Note that a smooth constrained optimization problem could have a moderate smoothness constant and at the same time have an arbitrary large Lipschitz constant.

Our solution uses the structure of an approximate proximal point method. The following related notions are important.  For a function $f$ and a regularization parameter $\newtarget{def:moreau_regularization_parameter}{\lambda} > 0$, the Moreau envelope is defined as
\[
    \newtarget{def:Moreau_Yoshida_envelope}{\M[\lambda f]}(x) \defi \inf_{y} \left\{ f(y) + \frac{1}{2\lambda} \norm{y-x}_2^2 \right\}.
\] 
We call the subproblem in the definition of $\M[\lambda f](x)$ the proximal subproblem, and denote $\newtarget{def:prox_point}{\prox[\lambda f]}(x) \defi \argmin\left\{ f(y) + \frac{1}{2\lambda} \norm{y-x}_2^2 \right\}$. If $f$ and $\lambda$ are clear from context, we will simply use $\M(x)$, $\prox[](x)$. If the proximal subproblem has a unique solution, then the Moreau envelope is differentiable and it is 
\[
\nabla \M(x) = \frac{1}{\lambda} (x - \prox[](x)),
\] 
see \citep{bertsekas2003convex}.

\section{RELATED WORK}

\citet{hardt2018gradient} %
introduced the quasar convex class, where it was named weak quasi-convexity. They provided an analysis of stochastic gradient descent for a weakly smooth quasar convex objective.
\citet{guminov2023accelerated} %
proposed an accelerated algorithm for smooth quasar convex functions with a search over a $2$-dimensional affine space. Their solution was inspired by the first nearly accelerated first-order method for convex smooth functions \citep{nemirovski1982orth,nemirovski_acceleration} and also by \citet{narkiss2005sequential}, that builds on the former. %
\citet{nesterov2021primal} %
followed up on the previous work, providing a universal algorithm that simplified the plane search to a line search and allowed for affine constraints. Because the feasible set is still an affine space, no degree of freedom is lost with respect to the unconstrained case. 
Later, \citet{hinder2020nearoptimal} %
also proposed an accelerated algorithm with line search, quantifying the time that it would take for a binary search to run, which required bounding the algorithm's iterates, and they proved nearly matching lower bounds, and coined the term quasar-convexity. They also studied strong quasar-convex problems.

\citet{wang2023continuized} extended the continuized approach for acceleration of \citet{even2021continuized} in order to obtain a randomized method that with high probability obtains an $\epsilon$ minimizer in $\bigo{\sqrt{\frac{LR^2}{\epsilon}}}$ iterations, that is, without $\log$ factors. 
\citet{lezane2024accelerated} generalized accelerated solutions to non-Euclidean and weak smoothness function classes. \citet{jin2020convergence} studied gradient norm minimization under quasar-convexity.
\citet{gower2021sgd,vaswani2022towards,fu2023accelerated} studied the optimization of stochastic quasar convex functions with adaptivity guarantees. %
\citet{saad2025new} provided lower bounds for stochastic quasar-convex functions. 
When the problem is Lipschitz \emph{non-smooth}, standard regret-minimization algorithms generalize to work for the best iterate with constrained settings for quasar convex problems, see \citet{joulani2020modular}.

Regarding the applicability of quasar-convexity, \citet{hardt2018gradient} show that several linear dynamical systems optimization tasks are quasar-convex.
\citet{foster2018uniform} show that \GLM{s}, that use losses of the form $w \mapsto (\sigma(\innp{w, x} -y))$ for a datapoint $(x, y)$, are quasar convex if the link function $\sigma$ satisfies boundedness conditions on its first two derivatives, although they stated a weaker property. \citet{wang2023continuized} provide further examples of families of \GLM{s} that are quasar-convex. \citet{foster2018uniform} also showed that a robust linear regression problem is quasar convex. \citet{martinezrubio2022global} reduced geodesically-convex optimization in constant curvature Riemannian Manifolds to constrained tilted-convex problems in the Euclidean space, and therefore to quasar-convex problems.

Our main algorithm is built as an inexact accelerated proximal point method. Obtaining accelerated algorithms by designing an inexact accelerated proximal point method and implementing the approximate proximal subproblem with first-order methods is a technique that has yielded general and fruitful frameworks, cf. \citep{monteiro2013, %
lin2017catalyst, %
frostig2015un-regularizing, %
ivanova2021adaptive, %
carmon2022recapp}, among many others.
A few works exploit the structure of some non-convex functions, such as weak convexity, in order to obtain a convex or strongly convex subproblem for the proximal point method, the first one of which could be \citep[Proposition 1]{kaplan1998proximal}. In a similar vein, \citep{davis2019stochastic} used the proximal point method with a regularizer value guaranteeing good properties for the proximal subproblem, and provide guarantees of stationarity for stochastic algorithms.

\citet{millan2025frank} developed an analysis showing rates for the Frank-Wolfe method under smoothness and star convexity, a special case of our problem setting.

\section{THE ACCELERATED METHOD}
Accelerated optimization for first-order methods in convex smooth optimization is an extensively studied topic. Initially considered an algebraic trick devoid of intuition, Nesterov's accelerated gradient descent \citep{nesterov1983method} has been explained and generalized via many efforts and interpretations yielding powerful frameworks \citep{nesterov2005smooth,monteiro2013,allenzhu2017linear,levy2018online,joulani2020simpler,wang2024noregret}. One point of view is that of linear coupling \citep{allenzhu2017linear} where the authors reinterpret an optimal algorithm like \citet[Section 3]{nesterov2005smooth} as a combination of gradient descent and an online learning algorithm like mirror descent, where the function value progress of gradient descent compensates the per-iterate regret of mirror descent. This analysis is done separately for unconstrained and constrained cases. 

Mirror descent or FTRL online-learning algorithms, cf. \citep{orabona2019modern}, on quasar-convex functions suffer a greater regret with respect to the convex counterpart, by a factor depending on $1/\gamma$. In the \emph{unconstrained} case this remains manageable: both the progress from gradient descent and the regret that one wants to compensate, still become proportional to the squared norm of the gradient at the currently computed point. A suitable coupling between the gradient descent and FTRL points can balance the proportionality constants and restore acceleration, provided one does a low-dimensional search \citep{hinder2020nearoptimal,guminov2023accelerated,nesterov2021primal}. In the \emph{constrained} case, however, this symmetry breaks: the usual descent proxy does not seem to necessarily compensate the increased regret, and a direct linear-coupling proof does not seem to go through. We lose a degree of freedom on how to choose the coupling.

Our remedy is an inexact implementation of the proximal point step, that is, computing an inexact version of $\prox(x)$. This approach offers the possibility of greater descent in exchange for greater computational expense in the step implementation. And further, one step of an approximate proximal point can be seen as an approximate gradient descent step on the associated Moreau envelope, with the property that this gradient descent step always lands in the feasible set $\X$ \citep{bauschke2011convex}, analogously to the unconstrained case. This leads us to investigate the properties of accelerated proximal point methods and the associated Moreau envelope for quasar convex smooth problems.

The proximal subproblem of a $\gamma$-quasar convex function $f$ is not necessarily quasar convex in general, not even if we allow for a different resulting quasar-convex constant. Indeed, since a center $x^\ast$ of a quasar convex function is a minimizer, it is enough to consider some quasar convex function plus a small quadratic centered at some point $x_0$ that shifts the minimizer of the sum towards $x_0$ with respect to $x^\ast$. The function $f$ does not even necessarily satisfy star convexity around $x_0$, which may make the sum not quasar convex. Alternatively, see \cref{ex:quasar_convex_function} for a concrete counterexample, which we display in \cref{fig:non_quasar_after_regularization}. Hence, the quasar-convex class is not closed under addition, even if we allow the parameter $\gamma$ to change after the sum. We claim that this non-additivity is natural and holds more generally for relaxations of convexity. Indeed, \citet[Section 2.1.1]{nesterov2018lectures} showed that if a class $\mathcal{F}$ of functions $f: \R^d \to \R$ satisfies (A) $\nabla f(x^\ast) = 0$ implies $x^\ast$ is a global solution, (B) the class is closed under conic combinations and (C) affine functions belong to the class, then $\mathcal{F}$ is the class of convex functions. As a consequence, there is no class larger than the one of convex functions that satisfies (A) and (B) at the same time. Therefore, relaxations satisfying (A) like quasar convexity, star convexity, tilted convexity, etc. cannot be closed under conic combinations. Similarly, a classical reduction adds the regularizer $x \mapsto \frac{\epsilon}{R^2}\norm{x_0-x}_2^2$, where $R$ is a bound for the distance from $x_0$ to a minimizer in order to exploit some strong convexity. The above observation suggests that this reduction is unlikely to hold for many relaxations of convexity satisfying (A). Likewise, computing the proximal point operator for general $\lambda > 0$ may lead to subproblems that either lack a minimizer, fail to have a unique solution, or are non-convex.

\begin{figure}[h!]
\centering
        \scalebox{1}[1.2]{\includegraphics[width=\columnwidth]{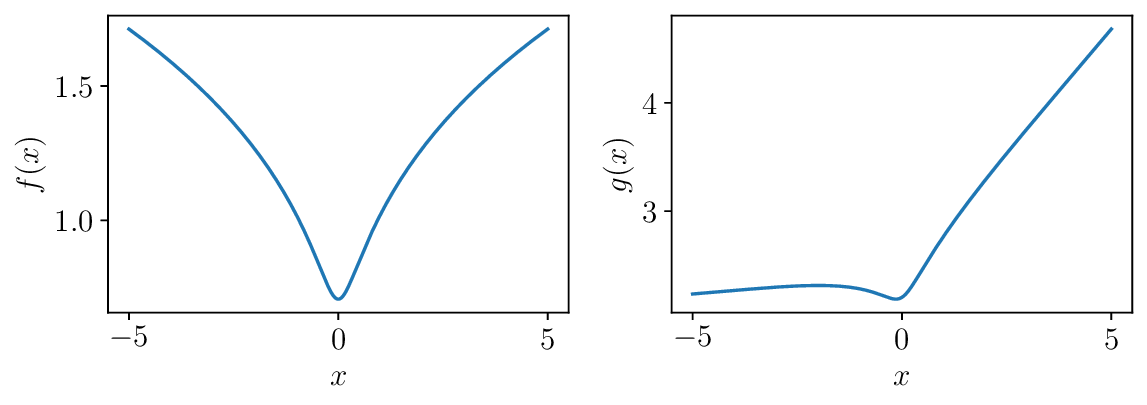}}
    \caption{\cref{ex:quasar_convex_function}; a $(1 / 2)$-quasar convex function $f(x) \defi (x^2+1/8)^{1 / 6}$ in $[-5, 5]$ (left) and its proximal subproblem $g(x) \defi f(x) + \frac{1}{2\lambda}(x-x_0)^2$, for $\lambda = 50$ and $x_0 \approx -12.23$ (right). The latter has a local maximizer at $x=-2$ and a local minimizer near $x=0$ and so it cannot be quasar convex in $[-5, 5]$.}
    \label{fig:non_quasar_after_regularization}
\end{figure}

The considerations above are negative results in the general case regarding combining functions for classes weaker than convexity. However, there are some special cases that we can exploit: by definition, if we add two functions that are $\gamma$-quasar \emph{with respect to the same center} $x^\ast$, then the sum becomes $\gamma$-quasar-convex. In particular, when adding $\indicator{\X}$ to $f$, provided $x^\ast\in\X$ and $f$ is $\gamma$-quasar-convex on $\X$ with respect to $x^\ast$, the composite function $f+\indicator{\X}$ preserves quasar-convexity around the same center. Similarly, for $(\gamma, \gamma_p)$-tilted convex problems any restriction to a compact convex set $\X$ results in a $\gamma$-quasar convex problem in $\X$. This is the case the problem appearing after the reduction from Riemannian optimization in \citep{martinezrubio2022global}. 
However, in order to exploit an inexact proximal point step, we further prove several properties of the Moreau envelope. If our original problem is $\L$-smooth, and if we choose $\lambda < 1 / \L$, we obtain a convex proximal subproblem where the strong convexity from the regularizer makes the subproblem be strongly convex with $O(1)$ condition number, as the following lemma shows.

\begin{lemma}\label{lemma:properties_of_prox_and_moreau_env_smoothness}\linktoproof{lemma:properties_of_prox_and_moreau_env_smoothness}
    Let $f:\Rd\to\R$ have $\L$-Lipschitz gradients on $\X$, and $\lambda = 1/(2\L )$. For any $x \in \X$, the proximal problem $\min_{y\in\X} f(y) + \frac{1}{2\lambda}\norm{y-x}_2^2$ is $\L$-strongly convex, $3\L $-smooth, and the associated Moreau envelope $\M[\lambda f + \indicator{\X}]$ is $2\L $-smooth.
\end{lemma}

The property above on the Moreau envelope and the subproblems allows to solve proximal subproblems fast and to approximate Moreau envelope's gradients and function values: it would take $\bigotilde{1}$ iterations of projected gradient descent (\newtarget{def:acronym_projected_gradient_descent}{\PGD{}}), cf. \cref{eq:projected_gradient_descent_update}, to achieve a $\delta$-minimizer of the subproblem \citep[Section 3.4.2]{bubeck15convexmonograph}. Now, a key property is that at least when the proximal subproblem has a unique solution, as guaranteed by \cref{lemma:properties_of_prox_and_moreau_env_smoothness},  the Moreau envelope inherits the $\gamma$-quasar-convexity from $f$. 

\begin{proposition}\label{lemma:quasar_convexity_of_moreau_env}\linktoproof{lemma:quasar_convexity_of_moreau_env}
    Let $f:\Rd\to\R$ be $\L$-smooth $\gamma$-quasar convex in a closed convex set $\X$ with respect to a center $x^\ast \in \X$. For any $\lambda < \frac{1}{\L}$, the Moreau envelope $\M[\lambda f  + \indicator{\X}]$ is $\gamma$-quasar convex with respect to $x^\ast$.
\end{proposition}

\begin{algorithm}
    \caption{Constrained Quasar-Cvx Acceleration}
    \label{alg:constrained_quasar_cvx_acceleration}
\begin{algorithmic}[1]
    \REQUIRE $\gamma$-quasar convex $\L$-smooth $f$ in a compact convex set $\X$, accuracy $\epsilon$, initial point $x_0$, tolerance $\delta$.
    \State $y_0 \gets z_0 \gets x_0$; 
    \State $\lambda \gets 1 / (2\L )$ 
    \State $T\gets \frac{4}{\gamma}\sqrt{\frac{\L\D^2}{\epsilon}}$
    \vspace{0.1cm}
    \hrule
    \vspace{0.1cm}
       \FOR{$ t \gets 1 $ \textbf{to} $T$} 
           \State $a_t \gets \gamma^2 t / (8\L )$; \ $A_t{=}\sum_{i=1}^t a_i = \gamma^2 t(t+1) / (16 \L)$
           \State $x_{t}{\gets}\hyperref[alg:binary_search]{\operatorname{BinaryLineSearch}}(f, y_{t-1}, z_{t-1},\L, \delta, \frac{A_{t-1}\gamma}{a_t})$  \label{line:call_to_bin_search}
           \State $y_{t} \in \argmindelta[y\in\X] \set{ f(y) + \frac{1}{2\lambda} \norm{x_t - y}_2^2 }$ \label{line:inexact_prox}\Comment{Using, for instance, \newlink{def:acronym_projected_gradient_descent}{{\color{mygray}PGD}} \cref{eq:projected_gradient_descent_update}.} 
           \State $\gradmtilde(x_t) \gets \frac{1}{\lambda} (x_t - y_t)$ \label{line:gradmtilde}
           \State $z_{t}\gets\argmin_{z\in\X}\set{\sum_{i=1}^{t}\frac{a_i}{\gamma}\innp{\gradmtilde(x_{i}), z}{+}\frac{\norm{z-x_0}_2^2}{2}}$ \label{line:acc_ftrl} \ENDFOR
        \State \textbf{return} $\hat{y}_T \in \argmindelta[y\in\X] \set{ f(y) + \frac{1}{2\lambda} \norm{y_T - y}_2^2 }$.
\end{algorithmic}
\end{algorithm}

Since we are only able to approximate the Moreau envelope, we have to deal with errors in our algorithm design and analysis. For the inexact Moreau envelope function value and gradient, omitting the notation's dependence on $\lambda$ and $f$, we use:
\[
    \newtarget{def:approximate_moreau_envelope}{\mtilde}(x)\defi f(y_x) + \tfrac{1}{2\lambda}\|y_x-x\|_2^2
 \ \ \text{and}\ \
\newtarget{def:approximate_gradient_of_moreau_envelope}{\gradmtilde}(x)\defi \tfrac{1}{\lambda}(x-y_x),
\]
where $y_x$ is any $\delta$-minimizer of $y\mapsto f(y)+\tfrac{1}{2\lambda}\|y-x\|_2^2$ over $\X$. We note that if we computed an exact minimizer, these notions would correspond to the exact function values and gradients for the envelope $\M$.

We set the total number of iterations to $T\defi\tfrac{4}{\gamma}\sqrt{\tfrac{\L\D^2}{\epsilon}}$ and fix the inner accuracy at $\delta \defi \L\D^2 / (10T^6)$ and thus the time taken for \PGD{} \cref{eq:projected_gradient_descent_update} to compute $\delta$-minimizers of proximal subproblems, like $y_t$ in Line \ref{line:inexact_prox} of \cref{alg:constrained_quasar_cvx_acceleration}, is nearly a constant
\[
\bigo{\log\big(\frac{\L\norm{x-\prox(x)}_2^2}{\delta}\big)} = \bigo{\log\big(\frac{\L\D^2}{\epsilon}\big)}=\bigotilde{1}.
\]
There is a simple stopping criterion that guarantees we achieved $\delta$-optimality, see \cref{lemma:stopping_criteria_of_pgd_subroutine_for_prox} in \cref{app:binary_search}.

Our accelerated \cref{alg:constrained_quasar_cvx_acceleration} uses three sequences:
(i) $y_t$ are approximate proximal points with parameter \(\lambda=1/(2\L )\);
(ii) $z_t$ are iterates from FTRL that control regret on the approximate envelope gradients; and
(iii) $x_t\in\conv{y_{t-1},z_{t-1}}$ are coupling points chosen by a noisy binary search (\cref{alg:binary_search}), where the noise comes from our inexactness to compute the proximal point and so it is bounded but otherwise it can be arbitrary.

The following lemma is an approximate version to the descent lemma of gradient descent, where the progress made after the inexact gradient descent step ${y_t = x_t - \lambda \gradmtilde(x_t)}$ is proportional to the squared norm of the inexact gradient, up to an additive error.
\begin{lemma}\label{lemma:approximate_descent_lemma}\linktoproof{lemma:approximate_descent_lemma}
    For $\lambda = 1 /(2\L )$, and $x_t, y_t$ from \cref{alg:constrained_quasar_cvx_acceleration}, let $\mtilde(x_t) \defi f(y_t) + \frac{1}{2\lambda}\norm{y_t - x_t}_2^2$ and  $\mtilde(y_t) = f(\hat{y}_t) + \frac{1}{2\lambda}\norm{\hat{y}_t - y_t}_2^2$, for $\hat{y}_t \in \argmindelta[y\in\X] \set{f(y) + \frac{1}{2\lambda}\norm{y - y_t}_2^2}$. Then
    \[
   \mtilde(y_t) - \mtilde(x_t) \leq  - \frac{1}{8\L }\norm{\gradmtilde(x_t)}_2^2 + \delta.
    \] 
\end{lemma}

Our accelerated solution also requires to find a point $x_t \in \conv{y_{t-1}, z_{t-1}}$ such that an inequality approximately like convexity between $y_{t-1}$ and $x_t$ holds for $M$. If we had convexity of $M$ and $x_t = \alpha y_{t-1} + (1-\alpha) z_{t-1}$, for some $\alpha \in [0, 1]$, we would have $\innp{\nabla M(x_t), x_t - z_{t-1}} = \frac{\alpha}{1-\alpha}\innp{\nabla M(x_t), y_{t-1} - x_t} \leq \frac{\alpha}{1-\alpha} (M(y_{t-1}) - M(x_t))$ and varying $\alpha\in [0, 1]$ we could get any positive weight on the right hand side. Instead, we guarantee a similar property on $\mtilde$ by performing a binary search over the segment $\conv{y_{t-1}, z_{t-1}}$, but we only have access to noisy evaluations of the function and gradient, which could potentially make the binary search branch to the wrong path. However, we keep an invariant that, by using smoothness of the objective, guarantees the maintained interval in the search contains an interval satisfying the termination condition that is large enough to guarantee early termination. The oracles $h,\hat h$ required by \cref{alg:binary_search} are implemented via $\mtilde(\cdot)$ and $\gradmtilde(\cdot)$, respectively.

\begin{algorithm}
    \caption{Binary Line Search($f$, $y$, $z$, $\L$, $\delta$, $c$)}
    \label{alg:binary_search}
\begin{algorithmic}[1]
    \REQUIRE $\L$-smooth function $f$ with domain $\X$, points $y, z$, smoothness constant $\L$, tolerance $\delta$, constant $c \geq 0$. Define $\lambda \defi 1 / (2\L )$, $g(\alpha) \defi \M[\lambda f + \indicator{\X}](\alpha y + (1-\alpha)z)$. Assume access to $h(\alpha), \hat{h}(\alpha) \in \R$ satisfying
    $ g(\alpha) \leq h(\alpha) \leq g(\alpha) + \delta_1$ and $\abs{g'(\alpha) - \hat{h}(\alpha)}  \leq \delta_2$. Target error: $\tildeepsilon$, cf. \cref{eq:delta-conds}.
    \vspace{0.1cm}
    \hrule
    \vspace{0.1cm}
    \State \textbf{if} $\hat{h}(1) \leq \tildeepsilon$ \textbf{then return} $1$ \label{line:first_check}
    \State \textbf{if} $h(1) - h(0) \geq - \delta_1$ \textbf{then return} $0$ \label{line:second_check}
    \State $a_0 \gets 0$; \quad $b_0 \gets 1$; \quad $\alpha \gets (a_0 + b_0) / 2$; \quad $k \gets 0$
    \WHILE {$\alpha \hat{h}(\alpha) > c (h(1)-h(\alpha)) + \tildeepsilon$}
        \IF{$h(\alpha) \geq h(1) - \delta_1$} 
        \State $a_{k+1} \gets \alpha$; \ \ $b_{k+1} \gets b_k$
        \ELSE
        \State $a_{k+1} \gets a_k$; \ \ $b_{k+1} \gets \alpha$ \label{line:update_of_ak_and_bk}
        \ENDIF
    \State  $k \gets k+1$
    \State $\alpha \gets (a_k + b_k) / 2$;
    \ENDWHILE
    \State \textbf{return} $\alpha y + (1-\alpha) z$
\end{algorithmic}
\end{algorithm}

\begin{lemma}\label{lemma:binary_search}\linktoproof{lemma:binary_search}
    With the notation of \cref{lemma:approximate_descent_lemma} and Line \ref{line:call_to_bin_search} of \cref{alg:constrained_quasar_cvx_acceleration}, and for any target constant $c\geq 0$, then \cref{alg:binary_search} returns a point $x_t$ satisfying
\begin{align*}
     \begin{aligned}
         \innp{\gradmtilde(x_t), x_t-z_{t-1}} &\leq c(\mtilde(y_{t-1}) - \mtilde(x_t)) \\
        &\ \ \ + \sqrt{8\L\D^2\delta} + (9 + 5c) \delta,
     \end{aligned}
\end{align*}
    and it computes $x_t$ after $\delta$-approximating a $\prox[\lambda f + \indicator{\X}]$ at most $\bigol{\log\left(\frac{\L\D^2}{\delta}\right)}$ times.
\end{lemma}
Since each $\delta$-approximation of a $\prox[\lambda f + \indicator{\X}]$ takes at most $\bigo{\log(\L\D^2 / \delta)}$ iterations for \PGD{}, Line \ref{line:call_to_bin_search} of \cref{alg:constrained_quasar_cvx_acceleration} can be implemented in $\bigotilde{1}$ first-order oracle queries. Using these tools, we obtain our near-optimal constrained accelerated method. The proof can be found in \cref{app:accelerated_proofs}.

\begin{theorem}\label{thm:acc_constrained_quasar_cvx}\linktoproof{thm:acc_constrained_quasar_cvx}
    Let $f:\Rd \to \R$ be $\L$-smooth and $\gamma$-quasar convex with center $x^\ast$ in a compact convex set $\X$ of diameter $\D$. \cref{alg:constrained_quasar_cvx_acceleration} obtains an $\epsilon$-minimizer in
    \[
        \bigotildel{\frac{1}{\gamma}\sqrt{\frac{\L\D^2}{\epsilon}}}
    \] 
    queries to a first-order oracle of $f$.
\end{theorem}

\section{UNACCELERATED METHODS}

In this section, we present an analysis for the convergence of the projected gradient descent method as well as for the Frank-Wolfe algorithm, for $\L$-smooth $\gamma$-quasar convex objectives with constraints. 

Projected gradient descent minimizes a quadratic model of the function, which is an upper bound on the function, if $\eta \leq 1 / \L$. The update rule, starting from a point $x$, is the following: %

\begin{align} \label{eq:projected_gradient_descent_update}
     \begin{aligned}
         x^+ &{\gets} \argmin_{y\in\X}\left\{ f(x) + \innp{\nabla f(x), y - x} + \frac{1}{2\eta}\norm{y-x}_2^2 \right\}  \\
         &= \proj{\X}(x - \eta \nabla f(x)).
     \end{aligned}
\end{align}
We note that the problem defining $x^+$ is strongly convex and thus the solution is unique. Define the gradient mapping with respect to $f$ as
\[
    \newtarget{def:gradient_mapping}{\gradmap{\X}}(x) \defi \frac{1}{\eta}(x-x^+).
\] 
We do not make the dependence on $\eta$ explicit in the notation since it will be clear from context.

Our analysis works via making use of the gradient mapping and its properties, which we state in the next two lemmas. These are known facts about the gradient mapping \citep{nesterov2004introductory} but we provide proofs in \cref{app:steepest_descent} for completeness.

\begin{lemma}\label{eq:reduced_gradient_for_cvxty_ineq}\linktoproof{eq:reduced_gradient_for_cvxty_ineq}
    Let $f:\R^d\to\R$ be a differentiable function in a closed convex set $\X$. For any $x, y \in \X$, we have
    \[
    \innp{\nabla f(x),x^+-y} \leq \innp{\gradmap{\X}(x),x^+ - y}.
    \] 
\end{lemma}

\begin{lemma}\label{eq:progress_by_reduced_gradient_norm}\linktoproof{eq:progress_by_reduced_gradient_norm}
    Let $\X$ be a closed convex set and let $f:\R^d\to\R$ be $\L$-smooth in $\X$. Starting at $x_t \in \X$, we compute $x_{t+1}$ from $x_t$ via \cref{eq:projected_gradient_descent_update} with $\eta = 1/\L$. Then
    \[
        f(x_{t+1}) - f(x_t) \leq - \frac{1}{2\L }\norm{\gradmap{\X}(x_t)}_2^2.
    \] 
\end{lemma}

Now we present our result for quasar convex objectives. A technical innovation in both projected gradient descent and the Frank-Wolfe analyses is that we devise a lower bound estimate of $f(x^\ast)$ that can be obtained from the algorithms, that differs from classical Frank-Wolfe analyses or analyses of \PGD{} using the gradient mapping. This bound allows to shift certain weights by some indices to allow for errors coming from quasar convexity to be canceled. 

Surprisingly, the structure of this new lower bound is very similar for \PGD{}'s and Frank-Wolfe's analyses. See \cref{app:proofs_unacc_methods}.  %

\begin{theorem}[Projected Gradient Descent Rate]\label{thm:gd_for_quasar_convex}\linktoproof{thm:gd_for_quasar_convex}
    Let $f:\Rd \to \R$ be $\L$-smooth and $\gamma$-quasar convex with center $x^\ast$ in a compact convex set $\X$. Let $\D \defi \diam(\setsuch{x \in \X}{f(x) \leq f(x_0)})$ %
    Then, for all $t \geq 1$, the iterates of projected gradient descent \cref{eq:projected_gradient_descent_update} with step size $\eta = 1/\L$ satisfy
    \[
        f(x_t) - f(x^\ast) \leq \frac{20\L\D^2}{(t+1)\gamma^2}.
    \] 
\end{theorem}

Now we present our result on the Frank-Wolfe algorithm, whose pseudocode can be seen in \cref{app:frank_wolfe}. This algorithm operates by solving linear subproblems rather than projections, which are quadratic problems.
The algorithm is not modified with respect to the classical algorithm in convex optimization, but our analysis differs to deal with the quasar convex assumption.
\begin{theorem}[Frank-Wolfe Rate]\label{thm:fw_quasar_analysis}\linktoproof{thm:fw_quasar_analysis}
    Let $f:\Rd \to \R$ be $\L$-smooth and $\gamma$-quasar convex with center $x^\ast$ in a compact convex set $\X$ of diameter $\D$. Then, for all $t \geq 1$, the iterates of the Frank-Wolfe \cref{alg:frank_wolfe} with both step sizes alternatives satisfy
    \[
        f(x_t) - f(x^\ast) \leq \frac{6\L\D^2}{(t+1)\gamma^2}.
    \] 
\end{theorem}

\begin{remark}
An interesting property of the result in \cref{thm:gd_for_quasar_convex,thm:fw_quasar_analysis} is that the algorithms can achieve their convergence rates without requiring any knowledge of the parameter $\gamma$. The fact that they adapt automatically to the unknown $\gamma$ highlights their robustness and makes the results easily applicable in practical machine learning settings.
\end{remark}

\section{CONCLUSION}

In this work, we addressed the open question \citep{martinezrubio2022global,lezane2024accelerated} of whether smooth quasar-convex functions can be efficiently optimized in the presence of general convex constraints. 

Prior work established accelerated rates in the unconstrained case, but these solutions did not work for the constrained case due to the loss of one degree of freedom in the algorithm's design.  Our main contribution consisted of closing this gap: we presented an accelerated first-order method that achieves an $\epsilon$-approximate solution in $\bigotildel{\frac{1}{\gamma}\sqrt{\frac{\L\D^2}{\epsilon}}}$ first-order oracle queries, and we further extended the scope of quasar-convex optimization beyond unconstrained domains analyzing unaccelerated algorithms. As a result of our accelerated method, we improved over prior Riemannian optimization algorithms via a reduction to quasar-convex optimization \citep{martinezrubio2022global}. 

Our results provide a new tools for analyzing and understanding algorithms in structured nonconvex optimization, which may shed further light on why simple gradient-based algorithms are so effective in large-scale machine learning.

\subsection*{Acknowledgements}

Project supported by a 2025 Leonardo Grant for Scientific Research and Cultural Creation from the BBVA Foundation.

\clearpage
\onecolumn

\printbibliography[heading=bibintoc] %

\clearpage
\appendix
\thispagestyle{empty}

%
\aistatstitle{Fast Quasar-Convex Optimization with Constraints: \\
Supplementary Materials}

\section{PROPERTIES}

\begin{proof}\linkofproof{lemma:properties_of_prox_and_moreau_env_smoothness}
    Since $f$ has $\L$-Lipschitz gradients, we have
    \[
        f(y) \geq f(x) + \innp{\nabla f(x), y-x} - \frac{\L}{2}\norm{x-y}_2^2 \text{ for all } x, y \in \X.
    \] 
    Thus, adding the $2\L $ strong convexity inequality for $\phi_z(y) = \L\norm{x-y}_2^2$ for any $z \in \Rd$, we obtain:
    \[
        (f+\phi_z)(y) \geq (f+\phi_z)(x) + \innp{\nabla (f+\phi_z)(x), y-x} + \frac{\L}{2}\norm{x-y}_2^2 \text{ for all } x, y \in \X.
    \] 
    So $f+\phi_z$ is $\L$-strongly convex. A similar argument shows that it is $3\L $-smooth, since $\phi_z$ has $2\L $-Lipschitz gradients.

    Finally, for $\lambda =1/(2\L )$, denote $M \defi \M[\lambda f + \indicator{\X}]$ for ease of notation.
   Note that since the proximal subproblem is strongly convex, we have by the first-order optimality condition that 
   \[
       \nabla f(\prox(x)) + \gradmap{\prox(x)} + 2\L  (\prox(x)-x) = 0,
   \] 
    where $\prox(x) \defi \argmin_{y\in\X} \set{f(y) + \frac{1}{2\lambda}\norm{y-x}_2^2}$ and $\gradmap{\prox(x)} \in \partial \indicator{\X}(\prox(x))$. Using this equation in $\circled{2}$ below and by the envelope theorem in $\circled{1}$, we have
    \begin{equation}\label{eq:grad_of_moreau_env}
        \nabla \M(x) \circled{1}[=] \frac{1}{\lambda}(x-\prox(x)) \circled{2}[=] \nabla f(\prox(x)) + \gradmap{\prox(x)}.
    \end{equation}

    Obtaining the smoothness of the Moreau envelope is standard: the proof does not rely on the convexity of $f$. Indeed, by definition, $\frac{1}{\lambda}$-smoothness means that the isotropic quadratic with leading coefficient $\frac{1}{2\lambda}$ whose zeroth- and first-order information coincide with that one of $\M$ at $x$, is above $\M$. The quadratic we are looking for takes the form $Q_y: \R^d \to \R, x\mapsto f(y)+\indicator{\X}(y) + \frac{1}{2\lambda}\norm{x-y}_2^2$ for $y=\prox(x)$, for which indeed $\M(x) = Q_y(x)$ and $\nabla \M(x) = \nabla Q_y(x)$. Moreover, for any $y$, it is $\M \leq Q_y$ by the definition of $\M$ as a $\min$. Alternatively, we can provide this argument in an algebraic form:
    \begin{align}
     \begin{aligned}
         \M(y) &\circled{1}[=] f(\prox(y))+\indicator{X}(\prox(y)) + \frac{1}{2\lambda}\norm{\prox(y)-y}_2^2 \\
         &\circled{2}[\leq] f(\prox(x)) +\indicator{X}(\prox(y)) + \frac{1}{2\lambda}\norm{\prox(x)-y}_2^2 \\
         & \circled{3}[=] \M(x) - \frac{1}{2\lambda} \norm{\prox(x)-x}_2^2 + \frac{1}{2\lambda}\norm{\prox(x)-y}_2^2 \\
         &\circled{4}[=] \M(x) - \frac{1}{\lambda}\innp{\prox(x)-x, y-x} + \frac{1}{2\lambda}\norm{y-x}_2^2 \\
         &\circled{5}[=] \M(x) + \innp{\nabla \M(x), y-x} + \frac{1}{2\lambda} \norm{x-y}_2^2,
     \end{aligned}
    \end{align}
        where $\circled{1}$ and $\circled{3}$ use the definition of the Moreau envelope and the prox, $\circled{2}$ holds by the optimality of $\prox(y)$, $\circled{4}$ is the cosine theorem and $\circled{5}$ uses the expression of $\nabla \M(x)$ the gradient that we computed in \cref{eq:grad_of_moreau_env}. 
    
\end{proof}

\begin{proof}\linkofproof{lemma:quasar_convexity_of_moreau_env}
    Recall that we have the following, for $\lambda = 1 / (2\L )$, cf. \cref{eq:grad_of_moreau_env}:
    \[
        \nabla \M(x) = 2\L (x-\prox(x)) = \nabla f(\prox(x)) + \gradmap{\prox(x)}.
    \] 
    Note that we can use the reasoning leading to \cref{eq:grad_of_moreau_env} since $f$ is assumed to be $\L$-smooth. 

    Using quasar convexity of $f$ and $\indicator{\X}$ (convexity implies $\gamma$-quasar convexity centered at a minimizer, $\gamma \in (0, 1]$, and for the convex $\indicator{\X}$, every point in $\X$ is a minimizer), we obtain:
    \begin{align} \label{eq:quasar_cvxty_of_M}
     \begin{aligned}
         \M(x^\ast) &= f(x^\ast) + \indicator{\X}(x^\ast) \\
         &\geq f(\prox(x)) + \frac{1}{\gamma}\innp{\nabla f(\prox(x)), x^\ast - \prox(x)}  + \indicator{\X}(\prox(x)) + \frac{1}{\gamma}\innp{\gradmap{\prox(x)}, x^\ast - \prox(x)} \\
         &= \left(\M(x) - \L\norm{x-\prox(x)}_2^2\right) + \frac{1}{\gamma}\innp{\nabla \M(x), x^\ast - x} + \frac{1}{\gamma}\innp{\nabla \M(x), x - \prox(x)}\\
         &= \M(x) + \frac{1}{\gamma}\innp{\nabla \M(x), x^\ast - x} + \L\left(\frac{2}{\gamma}-1\right) \norm{x-\prox(x)}_2^2\\
         &\geq  \M(x) + \frac{1}{\gamma}\innp{\nabla \M(x), x^\ast - x},
       \end{aligned}
    \end{align}
    where we used $\M(x^\ast) = f(x^\ast) = f(x^\ast) + \indicator{\X}(x^\ast)$ which holds since $\indicator{\X}\equiv 0$ in $\X$ and
    \[
        f(x^\ast) \circled{1}[\leq] f(\prox(x^\ast)) + \frac{1}{2\lambda}\norm{x^\ast - \prox(x^\ast)}_2^2 \circled{2}[=] \M(x^\ast) = \min_{y\in\X}\{f(y) +\frac{1}{2\lambda}\norm{x^\ast-y}_2^2 \} \circled{3}[\leq] f(x^\ast),
    \] 
    where $\circled{1}$ uses $x^\ast$ is a minimizer of $f$, $\circled{2}$ holds by definition of $\prox(\cdot)$, and $\circled{3}$ holds by substituting $y$ by $x^\ast$. Note that the inequality above also shows that $\prox(x^\ast) = x^\ast$ since the inequalities become equalities.
\end{proof}

\begin{lemma}\label{lemma:stopping_criteria_of_pgd_subroutine_for_prox}
    Let $F$ be an $\L$-strongly convex function with minimizer at $y^\ast$, that is also $3\L $-smooth as in our \cref{lemma:properties_of_prox_and_moreau_env_smoothness}. Then, if we have a point $y$ such that $\norm{\nabla F(y)}_2 \leq \sqrt{2\L \delta}$, then $F(y)-F(y^\ast) \leq \delta$. \PGD{} starting at distance at most $\D$ from $y^\ast$ obtains such a point in at most $\bigo{\log(\L\D^2 / \delta)}$ iterations.
\end{lemma}

\begin{proof}
    The following holds for a function $F$ with minimizer at $y^\ast$ that is $\newtarget{def:some_other_str_convexity_constant}{\mutilde}$-strongly convex and $\newtarget{def:some_other_smoothness_constant}{\Ltilde}$ smooth:
    \begin{equation}\label{eq:function_value_vs_grad_in_str_cvx_and_smooth}
        \frac{1}{2\Ltilde} \norm{\nabla F(y)}_2^2 \leq F(y)- F(y^\ast) \leq \frac{1}{2\mutilde}\norm{\nabla F(y)}_2^2. 
    \end{equation}
    For the function we consider, it is $\mutilde \gets \L$ and $\Ltilde \gets 3\L $. Thus, if we detect that $y$ satisfies $\norm{\nabla F(y)}_2 \leq \sqrt{2\L  \delta}$, then we have $F(y) - F(y^\ast) \leq \delta$. At the same time, we have this gradient norm guarantee whenever $F(y)-F(y^\ast) \leq \delta / 3$, by the first inequality above.

    Since \PGD{} starting at $y_0$ takes at most $\bigo{\frac{\Ltilde}{\mutilde}\log(\frac{\Ltilde\norm{y_0-y^\ast}_2^2}{\hat{\delta}})}$ iterations to obtain a $\hat{\delta}$-minimizer, cf. \citep[Section 3.4.2]{bubeck15convexmonograph} for instance, then for our function and $\hat{\delta} = \delta / 3$, we take $\bigo{\log\frac{\L\D^2}{\delta}}$ iterations to find a gradient with norm at most $\sqrt{2\L\delta}$ guaranteeing a function gap of at most $\delta$.
\end{proof}

\begin{example}\label{ex:quasar_convex_function}
We consider $f(x) = (x^2 + \tfrac18)^{1/6}$ with center $x^\ast = 0$ and show that it is $\frac{1}{2}$-quasar convex in $[-5, 5]$. Such a property means:
\[
    (x^2 + \frac{1}{8})^{1/6} - \frac{1}{8^{1/6}} = f(x) - f(x^\ast) \leq 2 x f'(x) = \frac{2}{3}(x^2 + \frac{1}{8})^{-5 / 6}x^2, \qquad \text{ for } x \in [-5, 5].
\]

Showing the inequality above is simple, by a change of variable $t = (x^2 + \frac{1}{8})^{1/6}$ (yielding $x^2 = t^6 - \frac{1}{8}$) gives that it is enough to show, for $t \in I \defi [\frac{1}{8^{1 /6}}, (25 + \frac{1}{8})^{ 1 / 6}]$, that the minimum of the following function is at least $0$:
\[
    g(t) \defi \frac{2}{3} t^{-5}(t^6 - \frac{1}{8}) - t + \frac{1}{8^{1 / 6}} \stackrel{?}{\geq} 0 \text{ for all } t \in I.
\] 
    This function is concave since $g''(t) = -\frac{5}{2}t^{-7} \leq 0$ and $I \subset \R_{> 0}$, and so it is enough to check the endpoints of the interval to look for the minimum, which indeed is nonnegative.

    It is direct to check that $g(x) \defi f(x) + \frac{1}{2\lambda}(x-x_0)^2$ for $\lambda = 50$ and a $x_0 \approx -12.23$, there is a local maximizer of $g$ at $x=-2$, by computing the value of $x_0$ that makes the derivative of $g$ be $0$ at $x = -2$. In that case $g''(-2) \leq 0$.
\end{example}

\section{ACCELERATED CONSTRAINED METHOD FOR SMOOTH QUASAR-CONVEX OPTIMIZATION}\label{app:accelerated_proofs}

Here we present the proof of our main theorem, the proofs of \cref{lemma:binary_search,lemma:approximate_descent_lemma}, that we use for the theorem, follow.

\begin{proof}\linkofproof{thm:acc_constrained_quasar_cvx}
    Throughout this proof, we use the value $\lambda = 1 /(2\L )$ and denote $\M(\cdot) \defi \M[\lambda f + \indicator{\X}](\cdot)$ and $\prox[](\cdot)$ its corresponding proximal operator, without making explicit the dependence on $\lambda$, $f$ and $\X$. Denote the diameter of the feasible set $\X$ by $\D$.

Denote $F_{k}(y) \defi f(y) + \frac{1}{2\lambda} \norm{y-x_k}_2^2$.
For $\delta > 0$ to be chosen later, assume that we optimize $F_k$ so that $\norm{y_k - y_k^\ast}_2 \leq \delta_1$, and so $F_k(y_k)- \delta_2 \leq F_k(y_k^\ast) = \M(x_k) \leq F_k(y_k)$, where $y_k^\ast = \argmin_y F_k(y) = \prox(x_k)$ for $\delta_2 = \delta$. Since the condition number of $F_k(y_k)$ is $\bigo{1}$, we can do so with \PGD{} in $\bigo{\log(\frac{\L\D^2}{\delta})}$ iterations, cf. \cref{lemma:stopping_criteria_of_pgd_subroutine_for_prox}.

We denote our approximation of $\nabla \M(x_k)$ by $\gradmtilde(x_k) \defi \frac{1}{\lambda} (x_k - y_k)$, as in Line \ref{line:gradmtilde} of \cref{alg:constrained_quasar_cvx_acceleration}. Also denote by $\mtilde(x_k) \defi f(y_k) + \frac{1}{2\lambda}\norm{y_k - x_k}_2^2$ the approximation of the Moreau envelope value.
Using the approximate optimality condition of $y_k$ and optimality of $y_k^\ast$, that is, $F_k(y_k) - F_k(y_k^\ast) \leq \delta$, we obtain
\begin{equation}\label{eq:approx_moreau_envelope_value}
    \mtilde(x_k)-\delta = F_k(y_k)-\delta \leq F_k(y_k^\ast) = \M(x_k) \leq F_k(y_k)
\end{equation}
\begin{equation}\label{eq:approx_gradient_innacuracy}
    \norm{\gradmtilde(x_k) - \nabla \M(x_k)}_2 \leq \frac{1}{\lambda}\norm{y_k - y_k^\ast}_2 \circled{1}[\leq] \sqrt{8\L \delta},
\end{equation}
where in $\circled{1}$ we that $F_k$ is $\L$-strongly convex, cf. \cref{lemma:properties_of_prox_and_moreau_env_smoothness}, and so $\frac{\L}{2\lambda^2}\norm{y_k -y_k^\ast}_2^2 \leq \frac{1}{\lambda^2}(F_k(y_k)- F_k(y_k^\ast)) \leq 4\L ^2\delta$.

    Let $a_t > 0$ for $t \geq 1$, whose value will be decided later, and $A_t \defi A_{t-1} + a_t = \sum_{i=1}^t a_i$. Denote $A_{0} = 0$ and $A_0\L _0 = 0$. We first build the following lower bound $L_t$ on $f(x^\ast) = \M(x^\ast)$, for $t \geq 1$:
\begin{align*}
     \begin{aligned}
         A_t \M(x^\ast) &\circled{1}[\geq] \sum_{i=1}^t a_i \M(x_i) + \frac{a_i}{\gamma}\innp{\nabla \M(x_i), x^\ast - x_i} \\
             & \circled{2}[\geq] (-A_t \delta + \sum_{i=1}^t a_i \mtilde(x_i)) + \sum_{i=1}^t \frac{a_i}{\gamma}\innp{\gradmtilde(x_i), x^\ast - x_i} \pm \frac{1}{2}\norm{x^\ast - x_0}_2^2 -\sqrt{8\L \delta}\frac{A_tD}{\gamma}  \\
             & \circled{3}[\geq] (-A_t \delta + \sum_{i=1}^t a_i \mtilde(x_i)) + \sum_{i=1}^t \frac{a_i}{\gamma}\innp{\gradmtilde(x_i), z_t - x_i} + \frac{1}{2}\norm{z_t - x_0}_2^2 \\
             & \quad - \frac{1}{2}\norm{x^\ast -x_0}_2^2 -\sqrt{8\L \delta}\frac{A_tD}{\gamma}  \\
             & \defi A_t L_t.
     \end{aligned}
\end{align*}
where above, $\circled{1}$ holds by the $\gamma$-quasar convexity of $M$. In $\circled{2}$, we used \cref{eq:approx_moreau_envelope_value}, and we also added and subtracted our approximate gradients, and used Cauchy-Schwartz along with \cref{eq:approx_gradient_innacuracy} and $x^\ast, x_i \in \X$ to obtain 
\[
    - \sum_{i=1}^t \frac{a_i}{\gamma} \innp{\nabla \M(x_i) - \gradmtilde(x_i), x^\ast - x_i} \geq -\sqrt{8 \L\delta}\frac{A_tD}{\gamma}.
\] 
We also added and subtracted half the initial squared distance to $x^\ast$ to use it in $\circled{3}$, where we bound some $x^\ast$ away by using the definition of $z_t$.
    Now, we define the gap $G_t \defi \mtilde(y_t) - L_t$. For $t \geq 1$, we are going to bound 
    \begin{equation}\label{eq:desideratum_one_step_acc}
      A_t G_t - A_{t-1}G_{t-1} - \eventindicator{t=1}\frac{1}{2}\norm{x_0-x^\ast}_2^2 \leq E_t,  
    \end{equation}
    for some value $E_t$, where $\newtarget{def:event_indicator_function}{\eventindicator{A}}$ denotes the event indicator that is $1$ if $A$ holds true and $0$ otherwise.  In the sequel, we use the notation $\mtilde(y_t) \defi f(\hat{y}_t) + \frac{1}{2\lambda}\norm{\hat{y}_t-y_t}_2^2$ for $\hat{y}_t \in \argmindelta[y] \{ f(y) + \frac{1}{2\lambda}\norm{y-y_t}_2^2 \}$. For the corresponding $\hat{y}_T$, and concatenating \cref{eq:desideratum_one_step_acc}, we obtain:
\begin{equation}\label{eq:finishing_adgt_argument}
    f(\hat{y}_T) - f(x^\ast)  \leq \mtilde(y_T) - \M(x^\ast) \leq G_T \leq \frac{1}{A_T}\left( \frac{1}{2}\norm{x^\ast-x_0}_2^2 + \sum_{i=1}^T E_i \right).
\end{equation}

Thus, we want $A_t$ to grow rapidly whereas $E_t$ should be moderate. We note that the extra term at $t=1$ in \cref{eq:desideratum_one_step_acc} is due to the lower bound $A_1\L _1$, which is not canceled by $A_0 L_0 \defi 0$, whereas for other values of $t$ this same term appears in both lower bounds $A_t L_t$ and $A_{t-1}L_{t-1}$, so they cancel. We now show \cref{eq:desideratum_one_step_acc}, for all $t \geq 1$:
\begin{align*}
\begin{aligned}
    A_t &G_t - A_{t-1}G_{t-1} - \eventindicator{t=1}\frac{1}{2}\norm{x_0-x^\ast}_2^2 \circled{1}[\leq] A_t(\mtilde(y_t)- \mtilde(x_t)) + \Ccancel[red]{a_t \mtilde(x_t)} + A_{t-1}(\mtilde(x_t) - \mtilde(y_{t-1})) \\
    & - \left( \Ccancel[red]{\sum_{i=1}^t a_i\mtilde(x_i)} + \sum_{i=1}^{t-1} \frac{a_i}{\gamma} \innp{\gradmtilde(x_i), z_t - x_i} + \frac{\norm{z_t - x_0}_2^2}{2}\right)  + A_t \delta  + \frac{A_tD}{\gamma}\sqrt{8\L \delta} - \frac{a_t}{\gamma}\innp{\gradmtilde(x_t), z_t - x_t} \\
    & + \left( \Ccancel[red]{\sum_{i=1}^{t-1} a_i\mtilde(x_i)} + \sum_{i=1}^{t-1} \frac{a_i}{\gamma} \innp{\gradmtilde(x_i), z_{t-1} - x_i} + \frac{\norm{z_{t-1} - x_0}_2^2}{2}\right)  - A_{t-1} \delta  - \frac{A_{t-1}\D}{\gamma}\sqrt{8\L \delta} \\
    &\circled{2}[\leq] \left(-\frac{A_t}{8\L }\norm{\gradmtilde(x_t)}_2^2 + A_t \delta \right) + \left( \frac{a_t}{\gamma}\innp{\gradmtilde(x_t), z_{t-1}-z_t} + \frac{a_t}{\gamma}\left(\sqrt{8\L\D^2\delta} + (9 + \frac{5a_t}{A_{t-1}}) \delta\right)\eventindicator{t=1}\right) \\
    & \ \ \ -\frac{1}{2}\norm{z_t - z_{t-1}}_2^2 + a_t\delta + \frac{a_t \D}{\gamma}\sqrt{8\L \delta}) \\
    &\circled{3}[\leq] \cancelto{{\color{blue}\leq 0}}{\left( \frac{a_t^2}{2\gamma^2} - \frac{A_t}{8\L } \right)\norm{\gradmtilde(x_t)}_2^2} + \frac{a_t}{\gamma}(\sqrt{8\L\D^2\delta} + (9 + \frac{5a_t}{\gamma A_{t-1}}) \delta)\eventindicator{t=1} + A_{t+1}\delta + \frac{a_t \D}{\gamma}\sqrt{8\L \delta} \\
    & \leq \frac{a_t}{\gamma}(\sqrt{8\L\D^2\delta} + (9 + \frac{5a_t}{\gamma A_{t-1}}) \delta)\eventindicator{t=1} +  A_{t+1}\delta + \frac{a_t \D}{\gamma}\sqrt{8\L \delta} \defi E_t.
\end{aligned}
 \end{align*}   
    where $\circled{1}$ holds by definition of $G_t$. In $\circled{2}$, for the first summand we apply \cref{lemma:approximate_descent_lemma}, and for the second summand we apply the guarantee of the binary line search, cf. \cref{lemma:binary_search} with $c = A_{t-1} \gamma/a_t $ after multiplying it by $a_t / \gamma$, and we merged the resulting term with the last one in the second line after $\circled{1}$. Also, in $\circled{2}$, we canceled some errors depending on $\delta$ and we also bounded the remaining terms in big parentheses after canceling terms, by noting that they consist of the subtraction $ - \ell_t(z_{t}) + \ell_t(z_{t-1})$ of a $1$-strongly convex function $\ell_t$, with minimizer $z_{t-1}$ by its definition, and so we upper bound it by $-\frac{1}{2}\norm{z_t - z_{t-1}}_2^2$ in $\circled{2}$. Now, in $\circled{3}$, we used Cauchy-Schwartz and Young's inequality $\innp{a, b} \leq \frac{1}{2}(a^2 + b^2)$ to cancel the terms depending on $z_{t-1}-z_t$ and note that the remaining terms proportional to $\norm{\gradmtilde(x_t)}_2^2$ are non-positive since the choice $a_t = \gamma^2 t/(8\L )$ makes $A_t / (8\L ) = \gamma^2 t(t+1) / (128 \L^2) \geq \gamma^2 t^2/(128\L ^2) = a_t^2 / (2\gamma^2)$.

    Thus, it only remains to bound the right hand side of \cref{eq:finishing_adgt_argument} with the value of $E_t$ above:

\begin{align*}
\begin{aligned}
    f(\hat{y}_T) - f(x^\ast) &\leq \frac{1}{A_T}\left( \frac{1}{2}\norm{x^\ast-x_0}_2^2 + \sum_{i=1}^T E_i \right) \circled{1}[\leq] \frac{8\L \norm{x^\ast - x_0}_2^2}{\gamma^2 T^2} + \frac{8\L }{\gamma^2 T^2}\sum_{i=1}^T E_i  \\
    &\circled{2}[\leq] \frac{\epsilon}{2} + \frac{\epsilon T}{16\L\D^2} \left[ T^2 (2\sqrt{8\L\D^2\delta} + (19 + 4T^2)\delta ) \right] \\
    &\circled{3}[\leq] \frac{\epsilon}{2} + \frac{\epsilon}{2} = \epsilon,
\end{aligned}
 \end{align*}
 where in $\circled{2}$ we used $T = \frac{4}{\gamma}\sqrt{\frac{\L\D^2}{\epsilon}}$, and we bounded $a_t \leq \gamma^2 T^2 / (8\L )$, $A_{t+1} \leq 4\gamma^2 T^2 / (16\L )$, $\gamma\leq 1$ and $a_t \leq 2A_{t-1}$ for $t \geq 2$. In $\circled{3}$, we used the value $\delta = \L\D^2 / (10 T^6)$.

Finally the total number of first-order oracle queries corresponds to $T$ times the complexity of the binary search, which is $\bigo{\frac{\L\D^2}{\epsilon}}$ by \cref{lemma:binary_search}, so the total complexity is $\bigotilde{\frac{1}{\gamma}\sqrt{\frac{\L\D^2}{\epsilon}}}$ as desired.
\end{proof}

\begin{proof}\linkofproof{lemma:approximate_descent_lemma}
We have the following, by adding and subtracting $\frac{\L}{2}\norm{y_t-x_t}_2^2$:
\begin{align*}
\begin{aligned}
    \mtilde(y_t) - \delta &\leq \M(y_t) \leq f(y_t) \pm \frac{\L}{2}\norm{y_t-x_t}_2^2  = \mtilde(x_t) - \frac{\L}{2}\norm{y_t-x_t}_2^2 = \mtilde(x_t) - \frac{1}{8\L }\norm{\gradmtilde(x_t)}_2^2.
\end{aligned}
\end{align*}
\end{proof}

\subsection{BINARY SEARCH}\label{app:binary_search}

\begin{proof}\linkofproof{lemma:binary_search} 
    Denote $\newtarget{def:intermediate_point_in_line_search}{\valpha} \defi \alpha y_{t-1} + (1-\alpha) z_{t-1}$.  For $f$ and $\lambda = 1 / (2\L )$, define the following function based on the Moreau envelope: $g(\alpha) = \M(\valpha)$, for some $t \geq 1$. We have that $g'(\alpha)={\innp{\nabla \M(\valpha), y_{t-1} - z_{t-1}}} = \frac{1}{\alpha}{\innp{\nabla \M(\valpha), \valpha - z_{t-1}}}$ is $\Lhat$-Lipschitz, with $\newtarget{def:smoothness_for_the_line_problem}{\Lhat} \defi 2\L  \norm{y_{t-1} - z_{t-1}}_2^2$. Indeed:
\[
    \abs{g'(\alpha_1) - g'(\alpha_2)} \leq \norm{\nabla \M(v_{\alpha_1})- \nabla \M(v_{\alpha_2})}_2\norm{y_{t-1}-z_{t-1}}_2 \leq 2\L \norm{y_{t-1}-z_{t-1}}_2^2,
\] 
    where in the last inequality we used the $2\L $-smoothness of $M$, cf. \cref{lemma:properties_of_prox_and_moreau_env_smoothness}. We will use $h(\alpha) \defi \mtilde(\valpha) \defi f(\walpha) + \frac{1}{2\lambda}\norm{\walpha - \valpha}_2^2$, for $\newtarget{def:prox_of_v_alpha}{\walpha} \defiin \argmindelta[w] \set{f(w) + \frac{1}{2\lambda}\norm{w - \valpha}_2^2}$, and $\hat{h}(\alpha) \defi \innp{\gradmtilde(\valpha), y_{t-1} - z_{t-1}} \defi \innp{\frac{1}{\lambda}(\valpha - \walpha), y_{t-1} - z_{t-1}}$ in order to implement the oracles in \cref{alg:binary_search}. So the number $\delta$-approximations of a $\prox[\lambda f + \indicator{\X}]$ in the lemma statement equals the number of times that we require access to $h(\alpha), \hat{h}(\alpha)$. By their definitions, we have:
\begin{equation}\label{eq:innacuracies_of_functions_in_a_segment}
    g(\alpha) \leq h(\alpha) \leq g(\alpha) + \delta_1  \text{ and } \abs{\hat{h}(\alpha) - g'(\alpha)} \leq \delta_2
\end{equation}
    where $\delta_1 = \delta$, cf. \cref{eq:approx_moreau_envelope_value}, and $\delta_2 = \sqrt{8\L \delta}\norm{y_{t-1}-z_{t-1}}_2 \leq \sqrt{8\L \delta}\D$ by using Cauchy-Schwartz on the expression resulting from substituting $g'(\alpha)$, $\hat{h}(\alpha)$ by their definitions, and applying \cref{eq:approx_gradient_innacuracy}, and $y_{t-1}, z_{t-1} \in \X$, $\diam(\X) = \D$.

We will show that we can do a binary search on $g$ up to certain accuracy by only having access to the adversarially noisy function values and derivatives $h$ and $\hat{h}$.

The stopping test is
\begin{equation}\label{eq:stop}
\alpha \hat h(\alpha) \le c\big(h(1)-h(\alpha)\big) + \tildeepsilon.
\end{equation}
for
\begin{equation}\label{eq:delta-conds}
    \newtarget{def:binary_search_accuracy}{\tildeepsilon} \defi \delta_2 + (9 + 5c) \delta_1,
\end{equation}
    where the output $x_t$ will be $\valpha = \alpha y_{t-1} + (1-\alpha) z_{t-1}$ for the final $\alpha$ that we compute. Note that ${\alpha \hat{h}(\alpha) =\innp{\gradmtilde(x_t), \alpha(y_{t-1} - z_{t-1})} = \innp{\gradmtilde(x_t), \valpha - z_{t-1}}}$.

    Define as well the $\tildeepsilon_{g} \defi \tildeepsilon-c\delta_1-\delta_2 = (9 + 4c)\delta_1$. If for some $\alpha$ we have
\begin{equation}\label{eq:true-criterion}
\alpha g'(\alpha)\ \le\ c\big(g(1)-g(\alpha)\big)\ +\ \tildeepsilon_{g},
\end{equation}
then \eqref{eq:stop} holds at $\alpha$. Indeed, by \cref{eq:innacuracies_of_functions_in_a_segment}: 
\[
    \alpha (\hat{h}(\alpha) - \delta_2) \leq \alpha g'(\alpha) \leq c(g(1)-g(\alpha)) + \tildeepsilon_g \leq c(g(1)-g(\alpha)) + \tildeepsilon_g + c\delta_1.
\] 

We first establish some invariants of our algorithm. If the check in Line \ref{line:first_check} of \cref{alg:binary_search} succeeds, that is $\hat{h}(1) \leq \tildeepsilon$, then \eqref{eq:stop} directly holds for $\alpha = 1$. The same is true if the second check succeeds, i.e., Line \ref{line:second_check} in \cref{alg:binary_search}. This is because then $\circled{1}$ below holds, for $\alpha =0$:
\[
    0\cdot \hat{h}(0) = 0 \circled{1}[\leq] c (h(1) - h(0)) + c\delta_1 \circled{2}[\leq]  c (h(1) - h(0)) + \tildeepsilon
\] 
and $\circled{2}$ holds by \cref{eq:delta-conds}. Note that this second check succeeds for the first iteration of \cref{alg:binary_search}. So from now on, we can assume 
\[
    \hat{h}(1) > \tildeepsilon \text{ and }  h(0) > h(1) + \delta_1.
\] 
We have the following invariants, for all $k \geq 0$, throughout the execution of \cref{alg:binary_search}:
\begin{enumerate}[label=(\arabic*)]
    \item $h(1) \geq h(b_k)$.
    \item $h(a_k) \geq h(1) - \delta_1 \geq h(b_k) - \delta_1$.
    \item $\hat{h}(b_k) \geq \tildeepsilon$.
\end{enumerate}

Indeed, (1) holds since it starts being $b_0 =1$ and after any update in Line \ref{line:update_of_ak_and_bk}, it is $h(b_k) < h(1) - \delta_1 \leq h(1)$. The first inequality of (2) holds at $k=0$ by the initial properties shown above and by any update made in Line \ref{line:update_of_ak_and_bk}, whereas the second inequality uses (1). For (3), we have that this property holds at $k=0$ by the initial properties shown above. For any time $k\geq 0$ for which $b_{k+1} \neq b_k$, we have that $b_{k+1} \in (0, 1]$ equals an $\alpha$ that did not trigger the termination of the while loop, and thus $\circled{1}$ below holds:
\[
    \hat{h}(b_{k+1}) \geq b_{k+1} \hat{h}(b_{k+1}) \circled{1}[>] c(h(1) - h(b_{k+1})) + \tildeepsilon \circled{2}[\geq] \tildeepsilon.
\] 
where $\circled{2}$ holds by the second inequality of (1).

We will now show by induction that throughout the execution of the algorithm, if we did not stop, the interval $[a_k, b_k]$ contains an interval of length
\begin{equation}\label{eq:def_Delta}
    \Delta \defi  \frac{(\tildeepsilon_{g}-\delta_1)}{(1+c/2)\Lhat} = \frac{8\delta_1}{\Lhat} > 0.
\end{equation}
such that \cref{eq:true-criterion} holds. 
For this, it is enough to show that there is a point $\alpha^\ast \in [a_k, b_k]$ such that $g(\alpha^\ast) = 0$ and $g(\alpha) \leq g(b_k)$. This is true since by $g(\alpha) \leq g(b_k) \leq h(b_k) \leq h(1) \leq g(1) + \delta_1$, and $\Lhat$-smoothness we have that for any $t \in [0, 1]$ with $|t-\alpha^\ast|\le \Delta / 2$:
\[
    g(1) - g(t) \geq g(\alpha^\ast) - \delta_1 - g(t) \geq -\frac{\Lhat}{2}(\alpha-t)^2 - \delta_1 \geq -\frac{\Lhat}{2}\abs{\alpha-t} - \delta_1 \geq -\frac{\Lhat\Delta}{4} - \delta_1,
\] 
\[
t g'(t) \leq |g'(t)|\le \Lhat|t-\alpha^\star|\le \Lhat\frac{\Delta}{2},
\]
and thus
\[
    t g'(t) \le c\big(g(1)-g(t)\big)+ \frac{\Lhat\Delta}{2}(1 + \frac{c}{2}) + \delta_1  \leq c\big(g(1)-g(t)\big) + \tildeepsilon_g.
\] 
So we only need to prove the existence of such an $\alpha^\ast$ inductively. Note that as long as we do not stop, the check of \cref{eq:stop} failed at the endpoints of $[a_k, b_k]$ so the good interval of length $\Delta$ would be completely in $[a_k, b_k]$.

We analyze two cases. By the invariant (2), these cases cover all possibilities. 
\paragraph{Case 1, $h(a_k)>h(b_k)+\delta_1$}
Recall that we have $\hat h(b_k)>\tildeepsilon$ by (3). Then
\[
g(a_k)\ge h(a_k)-\delta_1>h(b_k)\ge g(b_k),\qquad g'(b_k)\ge \hat h(b_k)-\delta_2> \tildeepsilon-\delta_2>0.
\]
Then, by continuity, there exists $\alpha^\ast\in(a_k,b_k)$ with $g'(\alpha^\ast)=0$ and $g(\alpha^\ast) \leq g(b_k)$. 

\paragraph{Case 2. $h(b_k) - \delta_1 \le h(a_k)\le h(b_k)+\delta_1$} 
Using (3) again, $g'(b_k)\ge \hat{h}(b_k)-\delta_2> \tildeepsilon-\delta_2 \geq \tildeepsilon_g > \Delta \Lhat$. Since we did not stop in the previous iteration and we halved the interval that inductively was of size greater than $\Delta$, we now have $b_k - a_k > \Delta / 2$. Define $\beta\defi b_k-\frac{\Delta}{2}\in(a_k,b_k]$ and note that by smoothness and \cref{eq:def_Delta}:
\[
    g'(\beta) \geq g'(b_k) - \frac{\Delta \Lhat}{2} > \frac{\Delta \Lhat}{2} > 0.
\] 
Applying smoothness again and using $g'(b_k) \geq \Delta \Lhat > 0$:
\[
g(\beta) \le g(b_k)-g'(b_k)\frac{\Delta}{2}+\frac{\Lhat}{2}\frac{\Delta^2}{4} \leq g(b_k) - \frac{3\Delta\Lhat}{8}
\circled{1}[\le]  g(b_k)-3\delta_1.
\]
where $\circled{1}$ holds by definition of $\Delta$, cf. \cref{eq:def_Delta}. Finally, using this last result, since $g\le h\le g+\delta_1$ and $h(a_k)\ge h(b_k) - \delta_1$, we have:
\[
g(\beta) \le g(b_k)-3\delta_1 \le h(b_k)-2\delta_1 \leq h(a_k)-\delta_1 \le g(a_k).
\]
We are now in the same situation as before. Since $a_k < \beta$, $g'(\beta) > 0$ and $g(a_k) \geq g(\beta)$, by continuity there must exist a point $\alpha^\ast \in (a_k, b_k)$ such that $g'(\alpha^\ast) = 0$ and $g(\alpha^\ast) \leq g(\beta) \leq g(b_k)$.

To conclude, the width of the interval $[a_k, b_k]$ was halved at each iteration, and the algorithm must stop before the length of this interval is $\leq \Delta$, so we successfully stop after at most
\[
\left\lceil \log_2 \left(\frac{1}{\Delta}\right)\right\rceil
=\left\lceil \log_2 \left(\frac{8\Lhat}{\delta_1}\right)\right\rceil \leq \left\lceil \log_2\left(\frac{8\L\D^2}{\delta_1}\right)\right\rceil
\]
iterations of the while loop. At each iteration we query $h$, $\hat{h}$ a constant number of times and before the loop we also do a constant number of queries. The total number of queries is $\bigo{\log(\frac{\L\D^2}{\delta})}$.
\end{proof}

\section{PROOFS FOR UNACCELERATED METHODS}\label{app:proofs_unacc_methods}

\subsection{Steepest Descent}\label{app:steepest_descent}

\begin{proof}\linkofproof{eq:reduced_gradient_for_cvxty_ineq}
  Define the quadratic $Q_x: \X\to \R$ as
  \begin{equation}
     Q_x(z)\defi f(x)+ \innp{\nabla f(x),z-x} + \frac{1}{2\eta}\norm{z-x}_2^2.
  \end{equation}
    For any $x$, define $x^+\defi \argmin_{z\in \X} Q_x(z)$ as the resulting point from taking a projected gradient descent step from $x$, cf. \cref{eq:projected_gradient_descent_update}. Note that $Q_x$ is $(1/\eta)$-strongly convex. Thus, for all $y\in \X$, we have
  \begin{align*}
      f(x) +& \innp{\nabla f(x),x^+-x} + \frac{1}{2\eta}\norm{x^+-x}_2^2 +\frac{1}{2\eta}\norm{x^+-y}_2^2 = Q_x(x^+)+\frac{1}{2\eta}\norm{x^+-y}_2^2 \\
      &\le Q_x(y) = f(x) + \innp{\nabla f(x),y-x} + \frac{1}{2\eta}\norm{y-x}_2^2.
  \end{align*}
  Or equivalently,
  \begin{align*}
      \innp{\nabla f(x),x^+-y}  &\leq  \frac{1}{2\eta}\norm{y-x}_2^2-\frac{1}{2\eta}\norm{x^+-x}_2^2-\frac{1}{2\eta}\norm{x^+-y}_2^2 \\ 
      &= \frac{1}{\eta}\innp{x-x^+, x^+-y} = \innp{\gradmap{\X}(x), x^+ - y}. 
  \end{align*}
\end{proof}

\begin{proof}\linkofproof{eq:progress_by_reduced_gradient_norm}
    Using the smoothness inequality, we get the result, after applying \cref{eq:reduced_gradient_for_cvxty_ineq} in $\circled{1}$:
    \begin{align*}
        f(x_{t+1}) - f(x_t) &\leq \innp{\nabla f(x_t), x_{t+1} - x_t} + \frac{\L}{2}\norm{x_{t+1} -x_t}_2^2 \\
        & \circled{1}[\leq] \innp{g_X(x_t), x_{t+1} - x_t} + \frac{1}{2\L }\norm{\gradmap{\X}(x_t)}_2^2 = - \frac{1}{2\L }\norm{\gradmap{\X}(x_t)}_2^2.
  \end{align*}
\end{proof}

\begin{proof}\linkofproof{thm:gd_for_quasar_convex}
    First, we argue that all the iterates stay in the set $\Y \defi \setsuch{x \in \X}{f(x) \leq f(x_0)}$ of diameter $\D$, which contains $x^\ast$ by definition. By induction, assuming $x_t$ is in the set, which holds by definition for the first iterate $x_0$, we show that $x_{t+1}$ is also in the set. The update rule implies:
    \[
        f(x_{t+1}) = \min_{x\in\X} \left\{ f(x_t) + \innp{\nabla f(x_t), x-x_t} + \frac{\L}{2}\norm{x-x_t}_2^2\right\}  \circled{1}[\leq] f(x_t) \circled{2}[\leq] f(x_0),
    \] 
    where $\circled{1}$ is obtained by setting $x=x_t$ and $\circled{2}$ is by the induction hypothesis. Thus, the property is proven.

    Now, for $t \geq 0$, let $a_t$ be weights to be determined and let $A_t = \sum_{i=0}^t = a_t$. Consequently, define $A_{-1} \defi 0$ and denote $g_t \defi \gradmap{\X}(x_t)$. For a minimizer $x^\ast \in \argmin_{x\in \X} f(x)$, we define the following lower bound $L_t$ on $f(x^\ast)$:

\begin{align*}
     \begin{aligned}
         A_t f(x^\ast) &\circled{1}[\geq] \sum_{i=0}^t a_i f(x_{i+1}) + \sum_{i=0}^t \frac{a_i}{\gamma} \Big(  \innp{\nabla f(x_{i}), x^\ast - x_{i+1}} + \innp{\nabla f(x_{i+1}) - \nabla f(x_i), x^\ast - x_{i+1}}  \Big) \\
         & \circled{2}[\geq]  \sum_{i=0}^t a_i f(x_{i+1}) + \sum_{i=0}^t \frac{a_i}{\gamma} \Big(  \innp{g_{i}, x^\ast - x_{i+1}} - \L\norm{x_{i+1}-x_i}_2\cdot\norm{x^\ast - x_{i+1}}_2  \Big) \\
         & \circled{3}[\geq]  \sum_{i=0}^t a_i f(x_{i+1}) - \sum_{i=0}^t \frac{2a_i}{\gamma} \norm{g_i}_2 \D \\  
         &\defi A_t L_t,
     \end{aligned}
\end{align*}
    where in $\circled{1}$ we applied a combination of the quasar-convexity inequalities given by points visited by the algorithm, and added and subtracted some terms so that in $\circled{2}$ we use \cref{eq:reduced_gradient_for_cvxty_ineq} on one term and Cauchy-Schwarz and gradient Lipschitzness of $f$ on the other. Finally in  $\circled{3}$ we applied Cauchy-Schwarz as well and used the bounds $\norm{x^\ast - x_{i+1}}_2 \leq \D$, which is due to the iterates staying in the set $\Y$.

    Define the gap $G_t = f(x_{t+1}) - L_t$. Our aim is to show $A_t G_t - A_{t-1} G_{t-1} \leq E_t$, for some value $E_t$ for all $t \geq 0$, so we can conclude $f(x_{t}) - f(x^\ast) \leq G_{t-1} \leq \frac{A_{t-2}G_{t-2} + E_{t-1}}{A_{t-1}} \leq \cdots \leq  \frac{1}{A_{t-1}}\sum_{i=0}^{t-1} E_i$. Let $D_t \defi \max_{i\in\set{0, \dots, t}}\norm{x^\ast - x_t}_2 \leq \D$. First, for $t \geq 1$, we have:
\begin{align*}
     \begin{aligned}
         A_t& G_t - A_{t-1} G_{t-1} = A_{t-1}(f(x_{t+1}) -f(x_{t}) ) + \Cbcancel[blue]{a_t f(x_{t+1})} \\
         & \ \ - \left( \Cbcancel[blue]{a_t f(x_{t+1})} + \Ccancel[red]{\sum_{i=0}^{t-1} a_i f(x_{i+1})} - \frac{2a_t}{\gamma}\norm{g_t}_2 \D + \Ccancel[blue]{\sum_{i=0}^{t-1} \frac{2a_i}{\gamma}\norm{g_i}_2 \D} \right)  \\
         & \ \ + \left(  \Ccancel[red]{\sum_{i=0}^{t-1} a_i f(x_{i+1})}  + \Ccancel[blue]{\sum_{i=0}^{t-1} \frac{2a_i}{\gamma}\norm{g_i}_2D} \right) \\
         &\circled{1}[\leq] - \frac{A_{t-1}}{2\L } \norm{g_t}_2^2 + \frac{2a_t}{\gamma}\norm{g_t}_2 \D \circled{2}[\leq] \Cbcancel[red]{- \frac{A_{t-1}}{2\L } \norm{g_t}_2^2} + \Cbcancel[red]{\frac{A_{t-1}}{2\L } \norm{g_t}_2^2} + \frac{2\L a_t^2}{\gamma^2 A_{t-1}}\D^2 \defi E_t.
     \end{aligned}
    \end{align*}
    where in $\circled{1}$ we used \cref{eq:progress_by_reduced_gradient_norm} and in $\circled{2}$ we used Young's inequality and canceled some terms.

    For $t=0$, the inequalities above hold up to before $\circled{2}$. Taking into account that $A_{-1}=0$, we have $A_0G_0 = A_0G_0 - A_{-1}G_{-1} = E_0 \leq 2a_0\L\D^2 /\gamma$, where we have $\norm{g_t}_2 \leq \L\D$ due to the fact that $x_t, x_{t+1}$ are in the level set of $x_0$.

    Finally, choosing $a_t=2t+2$ so that $A_t = (t+1)(t+2)$, we have, for all $t \geq 1$: 
    \begin{align*}
         \begin{aligned}
             f(x_t) - f(x^\ast) &\leq G_{t-1} \leq \frac{1}{A_{t-1}} \sum_{i=0}^{t-1} E_i = \frac{1}{A_{t-1}} \left( \sum_{i=1}^{t-1} \frac{2\L\D^2 a_i^2}{\gamma^2 A_{i-1}} + \frac{4\L\D^2}{\gamma} \right) \\
             &=  \frac{1}{\gamma^2 t(t+1)} \left(  \sum_{i=1}^{t-1} \frac{16\L\D^2(i+1)}{2i} + 4\gamma \L\D^2 \right) \leq \frac{(16+4\gamma/t)\L\D^2}{\gamma^2(t+1)} = \bigol{\frac{\L\D^2}{\gamma^2 t}}.
         \end{aligned}
    \end{align*}
    The numeric constant $20$ in the theorem statement comes from bounding the above using $\gamma \in (0, 1]$ and $t \geq 1$.
\end{proof}

\subsection{Frank-Wolfe}\label{app:frank_wolfe}

\begin{algorithm}[ht!]
    \caption{Frank-Wolfe algorithm}
    \label{alg:frank_wolfe}
   \begin{algorithmic}[1]
       \REQUIRE Function $f$ that is $\L$-smooth and $\gamma$-quasar convex in a compact convex set $\X$. Initial point $x_0 \in \X$. 
       \State $A_{-1} \gets 0$
       \vspace{0.2cm}
       \hrule
       \vspace{0.2cm}
       \FOR{$ t \gets 0 $ \textbf{to} $T-1$} 
           \State $a_t \gets 2t+2$; $A_t \gets A_{t-1} + a_t = \sum_{i=0}^t a_i = (t+1)(t+2)$
           \State $v_t \in \argmin_{v\in\X} \{\innp{\nabla f(x_{t}), v}\}$\label{line:computing_v_in_FW}
           \State {\color{mygray} Choose either Line \ref{line:step_size_gamma_dependent} or Line \ref{line:short_steps} for the whole execution}
           \State $\eta_t \gets \frac{a_{t} / \gamma}{A_{t-1} + a_t / \gamma}$\label{line:step_size_gamma_dependent}
           \State $\eta_t  = \min\left\{1,\frac{\innp{\nabla f(x_t),x_t-v_t}}{\L\norm{v_t-x_t}_2^2}\right\}$ \label{line:short_steps} \Comment{$= \argmin_{\eta\in[0,1]} \left\{\eta\innp{\nabla f(x_t),v_t-x_t} + \frac{L\eta^2}{2}\norm{v_t-x_t}_2^2\right\}$}
           \State  $x_{t+1} \gets (1-\eta_t)x_t + \eta_t v_t$ 
       \ENDFOR
        \State \textbf{return} $x_T$.
\end{algorithmic}
\end{algorithm}

\begin{proof}\linkofproof{thm:fw_quasar_analysis}
    We note that actually, this proof works for any norm as long as $\L$-smoothness and the diameter $\D$ are taken with respect to that norm. We prove here the theorem for the choice of the step size in Line \ref{line:computing_v_in_FW} where the one in Line~\ref{line:short_steps}, which is agnostic to $\gamma$, is \cref{prop:fw_short_step} below.
Let $a_t > 0$ to be determined later and define $A_t = A_{t-1} + a_t = \sum_{i=0}^{t} a_i$.
    Let $v_t \defiin \argmin_{v\in\X} \innp{\nabla f(x_t), v}$ and let $x_{t+1} \defi \frac{A_{t-1}}{A_{t-1}+a_{t} / \gamma}x_{t} + \frac{a_t / \gamma}{A_{t-1} + a_t / \gamma} v_{t}$ be defined as a convex combination of $x_{t}$ and $v_{t}$. Note $A_{-1}=0$ and $A_0 = a_0$, so $x_1 =v_0$. We define the following lower bound on $f(x^\ast)$ for a minimizer $x^\ast\in \argmin_{x\in \X} f(x)$:
    \begin{align*}
     \begin{aligned}
         A_t f(x^\ast)  &\circled{1}[\geq] \sum_{i=0}^t a_i f(x_{i+1}) + \sum_{i=0}^t \frac{a_i}{\gamma} \innp{\nabla f(x_{i+1}), x^\ast - x_{i+1}}  \\
         &\circled{2}[\geq]  \sum_{i=0}^t a_i f(x_{i+1}) + \sum_{i=0}^t \frac{a_i}{\gamma} \Big( \innp{\nabla f(x_{i}), v_{i+1} - x_{i+1}} + \innp{\nabla f(x_{i+1}) - \nabla f(x_{i}), v_{i+1} - x_{i+1}} \Big)  \\
         &\circled{3}[\geq]  \sum_{i=0}^t a_i f(x_{i+1}) + \sum_{i=0}^t \frac{a_i}{\gamma} \Big( \innp{\nabla f(x_{i}), v_{i+1} - x_{i+1}} - \L\norm{x_{i+1} - x_{i}}\cdot \D \Big) \defi A_t L_t, \\
     \end{aligned}
    \end{align*}
    where we applied quasar-convexity of $f$ in $\circled{1}$, the optimality of $v_{i+1}$ in $\circled{2}$, along with adding and subtracting some terms. And finally, we used Cauchy-Schwarz in $\circled{3}$ along with using gradient Lipschitzness of $f$, and also bounding $\norm{v_{i+1} - x_{i+1}}$ by $\diam(\X) = \D$.

    Now we define the gap $G_t \defi f(x_{t+1}) - L_t$. If we show $A_t G_t - A_{t-1} G_{t-1} \leq E_t$, for some number $E_t$ for all $t \geq 0$ (note $A_{-1} = 0$), then we will have $f(x_{t}) - f(x^\ast) \leq G_{t-1} \leq \frac{1}{A_{t-1}}\sum_{i=0}^{t-1} E_i$. Our aim is thus to have small $E_t$ and large $A_t$. We obtain the following, for $t \geq 0$:
    \begin{align*}
     \begin{aligned}
         A_t& G_t - A_{t-1} G_{t-1} = A_{t-1} (f(x_{t+1}) - f(x_{t})) + \Ccancel[red]{a_t f(x_{t+1})} \\
         & \ \ - \biggl( \Ccancel[red]{\sum_{i=0}^{t} a_i f(x_{i+1})} + \frac{a_t}{\gamma}(\innp{\nabla f(x_t), v_{t+1} - x_{t+1}} - \L\D\norm{x_{t+1} - x_t})    \\
         & \ \ + \Ccancel[blue]{\sum_{i=0}^{t-1} \frac{a_i}{\gamma}(\innp{\nabla f(x_i), v_{i+1} - x_{i+1}}   - \L\D\norm{x_{i+1} - x_i}} ) \biggr)\\
         & \ \ + \left( \Ccancel[red]{\sum_{i=0}^{t-1} a_i f(x_{i+1})} + \Ccancel[blue]{\sum_{i=0}^{t-1}  \frac{a_i}{\gamma}(\innp{\nabla f(x_i), v_{i+1} - x_{i+1}} - \L\D\norm{x_{i+1} - x_i} }) \right) \\
         &\circled{1}[\leq] \innp{\nabla f(x_t), A_{t-1}(x_{t+1} - x_{t}) - \frac{a_t}{\gamma}(v_{t+1} - x_{t+1})} + \frac{A_{t-1}\L}{2}\norm{x_{t+1}-x_t}^2   + \frac{a_t}{\gamma} \L\D\norm{x_{t+1}-x_t}  \\
         &\circled{2}[\leq] \frac{a_t}{\gamma}\innp{\nabla f(x_t), v_{t} -  v_{t+1}} + \frac{a_t^2 A_{t-1}\L\D^2}{2\gamma^2(A_{t-1} + a_t / \gamma)^2} + \frac{a_t^2 \L\D^2}{\gamma^2(A_{t-1}+a_t / \gamma)} \\
         &\circled{3}[\leq] \frac{3\L\D^2a_t^2}{2A_{t}\gamma^2} \defi E_t.  \\
     \end{aligned}
    \end{align*}
    In $\circled{1}$, we applied smoothness on the first term, and canceled and grouped some terms. In $\circled{2}$ we used that the definition of $x_{t+1}$ implies $A_{t-1} (x_{t+1}-x_t) = \frac{a_t}{\gamma} (v_{t}-x_{t+1})$ in order to obtain a term depending on $\nabla f(x_t)$ that is nonpositive by the optimality condition of the problem defining $v_t$. We also use that the definition of $x_{t+1}$ also implies $ x_{t+1} - x_t = \frac{a_t / \gamma}{A_{t-1} + a_t / \gamma} (v_t -x_t)$, so we use this to bound the other two distance terms and then bound $\norm{v_t -x_t}\leq \D$. In $\circled{3}$ we dropped the gradient term and we also used $\gamma \leq 1$ for the term $A_{t-1} + a_t / \gamma \geq A_t$ in two denominators. Note that all of these computations are valid for $t=0$.

    Finally, choosing $a_t = 2t+2$, we have $A_t = \sum_{i=0}^t a_i =  (t+1)(t+2)$, so for all $t \geq 1$:
\begin{align*}
     \begin{aligned}
         f(x_{t})- f(x^\ast)\leq G_{t-1}\leq\frac{1}{A_{t-1}}\sum_{i=0}^{t-1} E_i = \frac{1}{A_{t-1}}\sum_{i=0}^{t-1} \frac{3\L\D^2 a_i^2}{2 A_{i}\gamma^2}  = \frac{1}{t(t+1)}\sum_{i=0}^{t-1}\frac{6\L\D^2(i+1)}{(i+2)\gamma^2} < \frac{6 \L\D^2}{(t+1)\gamma^2}.
     \end{aligned}
    \end{align*}
\end{proof}

\begin{proposition}[Frank-Wolfe with short steps]\label{prop:fw_short_step}
    Let $f:\Rd\to\R$ be $\L$-smooth and $\gamma$-quasar convex with center $x^\ast$ in a compact convex set $\X$ of diameter $\D$. Then, for every $T\geq 1$, the iterates of \cref{alg:frank_wolfe} using the short step in Line~\ref{line:short_steps} satisfy
    \[
        f(x_T)-f(x^\ast)
        \leq
        \frac{2\L\D^2}{\gamma\bigl(2+\gamma(T-1)\bigr)}
        \leq
        \frac{2\L\D^2}{\gamma^2(T+1)}.
    \]
\end{proposition}

\begin{proof}
    For $t\geq1$, define $\widehat{A}_t\defi(t-1+2/\gamma)^2$, and for $t\geq2$, define
    \[
        \widehat{a}_t
        \defi
        \widehat{A}_t-\widehat{A}_{t-1}
        =
        2t+\frac{4}{\gamma}-3.
    \]
    Thus, $\widehat{A}_1+\sum_{i=2}^t\widehat{a}_i=\widehat{A}_t$. For every $t\geq1$, we define the following lower bound:
    \begin{align*}
        \widehat{A}_t f(x^\ast)
        &\circled{1}[\geq]
        \widehat{A}_1 f(x_1)
        +\frac{\widehat{A}_1}{\gamma}
        \innp{\nabla f(x_1),x^\ast-x_1}
        +\sum_{i=2}^t\widehat{a}_i f(x_{i-1})
        +\sum_{i=2}^t\frac{\widehat{a}_i}{\gamma}
        \innp{\nabla f(x_{i-1}),x^\ast-x_{i-1}}\\
        &\circled{2}[\geq]
        \widehat{A}_1 f(x_1)
        +\frac{\widehat{A}_1}{\gamma}
        \innp{\nabla f(x_1),v_1-x_1}
        +\sum_{i=2}^t\widehat{a}_i f(x_{i-1})
        +\sum_{i=2}^t\frac{\widehat{a}_i}{\gamma}
        \innp{\nabla f(x_{i-1}),v_{i-1}-x_{i-1}}
        \defi
        \widehat{A}_tL_t.
    \end{align*}
    Here $\circled{1}$ uses quasar convexity, and $\circled{2}$ uses $x^\ast\in\X$ and the optimality of the linear minimization oracle. Define $G_t\defi f(x_t)-L_t$. Since $L_t\leq f(x^\ast)$, we have $f(x_t)-f(x^\ast)\leq G_t$.

    If $\innp{\nabla f(x_0),x_0-v_0}=0$, then quasar convexity and the optimality of $v_0$ imply that $x_0$ is optimal, so assume $\innp{\nabla f(x_0),x_0-v_0}>0$. The initial potential is bounded as
    \begin{align*}
        \widehat{A}_1G_1
        &=
        \frac{\widehat{A}_1}{\gamma}
        \innp{\nabla f(x_1),x_1-v_1}
        =
        \frac{\widehat{A}_1}{\gamma}\left(
        \innp{\nabla f(x_1)-\nabla f(x_0),x_1-v_1}
        +\innp{\nabla f(x_0),x_1-v_1}
        \right)\\
        &\circled{1}[\leq]
        \frac{\widehat{A}_1}{\gamma}\left(
        \L\norm{x_1-x_0}_2\D
        +\innp{\nabla f(x_0),x_1-v_0}
        \right)\\
        &\circled{2}[=]
        \frac{\widehat{A}_1}{\gamma}\left(
        \L\eta_0\norm{v_0-x_0}_2\D
        +(1-\eta_0)\innp{\nabla f(x_0),x_0-v_0}
        \right)\\
        &\circled{3}[\leq]
        \frac{\widehat{A}_1\L\D^2}{\gamma}
        \bigl(\eta_0+\eta_0(1-\eta_0)\bigr)\\
        &\leq
        \frac{\widehat{A}_1\L\D^2}{\gamma}
        =
        \frac{4\L\D^2}{\gamma^3}.
    \end{align*}
    In $\circled{1}$, we used the $\L$-Lipschitzness of the gradient, $\norm{x_1-v_1}_2\leq\D$, and the optimality of $v_0$. In $\circled{2}$, the definition of $x_1$ implies
    \(
        x_1-x_0=\eta_0(v_0-x_0),
        x_1-v_0=(1-\eta_0)(x_0-v_0).
    \)
    In $\circled{3}$, we used $\norm{v_0-x_0}_2\leq\D$ and the short-step identity
    \[
        (1-\eta_0)\innp{\nabla f(x_0),x_0-v_0}
        =
        (1-\eta_0)\L\eta_0\norm{v_0-x_0}_2^2,
    \]
    which also holds when $\eta_0=1$. The final equality uses $\widehat{A}_1=4/\gamma^2$.

    For $t\geq2$, let
    \[
        \bar{\eta}_{t-1}
        \defi
        \frac{2}{\gamma(t-1)+2}
    \]
    be the $\gamma$-dependent step in Line~\ref{line:step_size_gamma_dependent}. The proof weights satisfy
    \[
        \widehat{A}_t\bar{\eta}_{t-1}-\frac{\widehat{a}_t}{\gamma}
        =
        \frac{1}{\gamma},
        \qquad
        \widehat{A}_t\bar{\eta}_{t-1}^2
        =
        \frac{4}{\gamma^2}.
    \]
    The potential difference is bounded by the single chain
    \begin{align*}
        \widehat{A}_t&G_t-\widehat{A}_{t-1}G_{t-1}
        =
        \widehat{A}_t\bigl(f(x_t)-f(x_{t-1})\bigr)
        +\frac{\widehat{a}_t}{\gamma}
        \innp{\nabla f(x_{t-1}),x_{t-1}-v_{t-1}}\\
        &\circled{1}[\leq]
        \widehat{A}_t\left(
        -\bar{\eta}_{t-1}\innp{\nabla f(x_{t-1}),x_{t-1}-v_{t-1}}
        +\frac{\L\bar{\eta}_{t-1}^2}{2}\norm{v_{t-1}-x_{t-1}}_2^2
        \right)
        +\frac{\widehat{a}_t}{\gamma}
        \innp{\nabla f(x_{t-1}),x_{t-1}-v_{t-1}}\\
        &=
        -\frac{1}{\gamma}\innp{\nabla f(x_{t-1}),x_{t-1}-v_{t-1}}
        +\frac{2\L}{\gamma^2}\norm{v_{t-1}-x_{t-1}}_2^2\\
        &\leq
        \frac{2\L\D^2}{\gamma^2}
        \defi
        E_t.
    \end{align*}
    In $\circled{1}$, the short step in Line~\ref{line:short_steps} gives the smallest smoothness upper bound over $\eta\in[0,1]$, so it is no larger than the bound obtained with $\bar{\eta}_{t-1}$, the choice in Line~\ref{line:step_size_gamma_dependent}. In the last inequality, we used the optimality of $v_{t-1}$ so that $\innp{\nabla f(x_{t-1}), v_{t-1}} \leq \innp{\nabla f(x_{t-1}), x_{t-1}}$ and $\norm{v_{t-1}-x_{t-1}}_2\leq\D$.
    Finally, for every $T\geq1$,
    \begin{align*}
        f(x_T)-f(x^\ast)
        &\leq
        G_T
        \leq
        \frac{1}{\widehat{A}_T}
        \left(
        \widehat{A}_1G_1
        +\sum_{t=2}^T E_t
        \right)
        \circled{1}[\leq]
        \frac{1}{(T-1+2/\gamma)^2}
        \left(
        \frac{4\L\D^2}{\gamma^3}
        +\frac{2\L\D^2(T-1)}{\gamma^2}
        \right) \\
        &=
        \frac{2\L\D^2}{\gamma\bigl(2+\gamma(T-1)\bigr)}
        \circled{2}[\leq]
        \frac{2\L\D^2}{\gamma^2(T+1)}.
    \end{align*}
    Here $\circled{1}$ uses the bounds on $\widehat{A}_1G_1$ and $E_t$ and the choice of $\widehat{A}_T$, and $\circled{2}$ uses $\gamma\leq1$.
\end{proof}

\end{document}